\documentclass[reqno]{amsart}
\usepackage{amsmath,amssymb,amsthm,amsrefs, amscd, bbold}
\usepackage{xcolor}
\usepackage[english]{babel}
\usepackage[utf8]{inputenc}
\usepackage{hyperref}
\usepackage[normalem]{ulem}

\newtheorem{thm}{Theorem}[section]
\newtheorem{pro}[thm]{Proposition}
\newtheorem{lem}[thm]{Lemma}
\newtheorem{cor}[thm]{Corollary}
\theoremstyle{definition}
\newtheorem{defi}[thm]{Definition}
\newtheorem{rem}[thm]{Remark}
\newtheorem{ex}[thm]{Example}

\newcommand{\vu}{\vspace{.1cm}}

\newcommand{\1}{\mathbb{1}}
\newcommand{\due}{\underline{\overline{\#}}}
\newcommand{\et}{\varepsilon_t}
\newcommand{\es}{\varepsilon_s}

\allowdisplaybreaks

\title[Matched pairs of Weak Hopf algebras]{Weak Hopf algebras arising from  weak matched pairs}

\author[Fonseca, Martini and Silva]{Graziela Fonseca, Grasiela Martini and Leonardo Silva}

\address[Fonseca]{Instituto Federal Sul-rio-grandense, Brazil.}
\email{grazielalangone@gmail.com}

\address[Martini]{IME, Universidade Federal do Rio Grande do Sul, Rio Grande do Sul, Brazil.}
\email{grasiela.martini@gmail.com}

\address[Silva]{IME, Universidade Federal do Rio Grande do Sul, Rio Grande do Sul, Brazil.}
\email{dsleonardo@ufrgs.br}

\begin{document}

	\allowdisplaybreaks
	\begin{abstract}
		This work extends the idea of matched pairs presented by Majid in \cite{Majid} and Takeuchi in \cite{Takeuchi} for the context of weak bialgebras and weak Hopf algebras.
		We introduce, also inspired by partial matched pairs \cite{matchedpair}, the notion of weak matched pairs and establish conditions for a subspace of the smash product be a weak bialgebra/Hopf algebra.
		Further, some new examples of (co)actions of weak bialgebras over algebras and some results about integral elements are presented.

	\end{abstract}

	
	
	\thanks{The second and third author were partially supported by FAPERGS (Brazil), projects n. 23/2551-0000803-3 and 23/2551-0000913-7, respectively.}
	
	\maketitle
	
	\tableofcontents
	
	\section{Introduction}
	A matched pair of Hopf algebras was introduced in the literature by Singer, and under hypothesis of commutativity and cocommutativity, a new Hopf algebra was obtained \cite{Singer}.
	Later, Takeuchi also dealt with matched pairs and generalized the previous construction, and obtained a new Hopf algebra which was called bismash product \cite{Takeuchi}.
	After these works, matched pairs became an interesting way to construct new algebraic structures. For instance, Majid presented several constructions and results that generated new structures of bialgebras/Hopf algebras, by replacing the commutativity and cocommutativity requested by Takeuchi for others compatibility conditions, and described a huge class of noncommutative and noncocommutative Hopf algebras, which have been called bicrossed product of Hopf algebras (for instance, see \cite{Majid}).
	Since then, the idea of matched pairs was generalized and carry on to several contexts, such as Lie algebras, groupoids and others (for instance, see \cites{Agore, Nicolas_intro, muller, Sonia_Natale, zhang}).
	
	\vu
	
	In this work, we investigate the construction of weak Hopf algebras through matched pairs.
	It is folkloric that (global) actions and coactions of weak Hopf algebras resemble partial actions and coactions of (usual) Hopf algebras \cite{Batista_why}, so that much of we do in this paper is inspired by \emph{partial matched pairs}, particularly in \cite{matchedpair}.
	
	\vu
	
	The present work is divided as follows. In Section 2, the concepts of weak bialgebras and weak Hopf algebras plus some useful properties, actions and coactions, and some examples, will be briefly recalled. 
	Such examples will appear throughout the work.
	In Section 3, we introduce the definition of \emph{weak matched pair} as a pair $(H, A)$ of two weak bialgebras plus some compatibility conditions (see Definition \ref{novo weak matched pair}).
	It generalizes the classical matched pairs defined by Majid and by Takeuchi (the latter when the bialgebra is a Hopf algebra).
	But, different from there, in the weak scenario the product smash $A\# H$ is not a (unital) algebra or a (counital) coalgebra.
	Then, in Section 4, we deal with this product smash and investigate conditions in the weak matched pair to provide a structure of weak bialgebra/Hopf algebra in a subspace $A \due H$ of $ A \# H$.
	At last, in Section 5, we present some weak Hopf algebras generated by weak matched pairs, and some results about integral elements.
	
	\vu
	
	Throughout, algebraic structures (vector spaces, algebras...) will be all considered over a fixed field $\Bbbk$. Unadorned $\otimes$ means $\otimes_{\Bbbk}$.
	
	\section{Preliminaries}\label{sec_weak_preliminares}
	
	\subsection{Weak Hopf algebras}
	
	We recall some classical definitions and results about weak Hopf algebras. For details and proofs we refer \cite{Gabintegral}.
	
	\begin{defi}\label{def_weakbialgebra}
		A \textit{weak bialgebra} $H$ is a vector space such that $ (H, m, u) $ is an algebra, $(H,\Delta,\varepsilon)$ is a coalgebra, and the following conditions are satisfied for all $h,k \in H$:
		\begin{enumerate}
			\item[(i)] $\Delta(hk) = \Delta(h)\Delta(k)$,
			\item[(ii)] $\varepsilon(hk\ell) = \varepsilon(hk_1)\varepsilon(k_2\ell) = \varepsilon(hk_2)\varepsilon(k_1\ell)$,
			\item[(iii)] $(1_H \otimes \Delta(1_H))(\Delta(1_H) \otimes 1_H) = (\Delta(1_H) \otimes 1_H)(1_H \otimes \Delta(1_H)) = \Delta^2(1_H)$.
		\end{enumerate}
	\end{defi}

	Notice that in the last item we use $\Delta(1_H)= (1_{H})_1 \otimes (1_{H})_2 = 1_{H1} \otimes 1_{H2}$ two times, so it means $$1_{H1} \otimes 1'_{H1} 1_{H2} \otimes 1'_{H2} = 1_{H1} \otimes 1_{H2} 1'_{H1} \otimes 1'_{H2} = 1_{H1} \otimes 1_{H2} \otimes 1_{H3}.$$
	
	It is possible to define the following two linear maps
	\begin{eqnarray*}
		\varepsilon_t: H &\rightarrow& H \ \ \ \ \ \ \  \ \ \ \ \mbox{ and} \ \  \ \ 	\varepsilon_s: H \rightarrow H \\
		h &\mapsto& \varepsilon(1_1h)1_2 \ \ \ \ \ \ \ \ \ \ \ \  \ \ \ \  \ \ \	h \mapsto 1_1\varepsilon(h1_2)
	\end{eqnarray*}
	and so the vector spaces $H_t= \varepsilon_t(H)$ and $H_s= \varepsilon_s(H)$.
	Thus, for any weak bialgebra $H$, every element $ h \in H $ can be written as
	\begin{eqnarray*}
		h = (\varepsilon \otimes id) \Delta(h) 
		= (\varepsilon \otimes id) \Delta(1_H h) 
		= \varepsilon_t(h_1)h_2,\\
		h = (id \otimes \varepsilon) \Delta(h) 
		= (id \otimes \varepsilon) \Delta(h1_H) 
		= h_1 \varepsilon_s(h_2). 
	\end{eqnarray*}
	
	\begin{pro}
		Let $H$ be a weak bialgebra. Then, the following properties hold for all $h, k \in H$:
		\begin{eqnarray}
			\varepsilon(h\varepsilon_t(k)) &=& \varepsilon(hk)  \label{4.5} \\
			\varepsilon(\varepsilon_s(h)k) &=& \varepsilon(hk)  \label{4.6} \\
			\Delta(1_H) &\in& H_s \otimes H_t \label{4.7} \\
			h_1 \otimes \varepsilon_t(h_2) &=& 1_1h \otimes 1_2  \label{4.12}\\
			\varepsilon_s(h_1) \otimes h_2 &=& 1_1 \otimes h1_2 \label{4.13} \\
			h\varepsilon_t(k) &=& \varepsilon(h_1k)h_2  \label{4.14} \\
			\varepsilon_s(h)k &=& k_1\varepsilon(hk_2)  \label{4.15}\\
			\varepsilon_t(\varepsilon_t(h)k) &=& \varepsilon_t(h)\varepsilon_t(k)  \label{4.19} \\
			\varepsilon_s(h\varepsilon_s(k)) &=& \varepsilon_s(h)\varepsilon_s(k)  \label{4.20}
		\end{eqnarray}
		Furthermore, if $H$ is a commutative weak bialgebra, then both $\varepsilon_t$ and $\varepsilon_s$ are multiplicative maps.
	\end{pro}
	
	\begin{pro}[\cite{Caenepeel_}] \label{equivalencia_doum}Let $ H $ be a weak bialgebra. Then, for all $h\in H$, 
		\begin{itemize}
			\item[(i)] $\Delta^2(1_H)=1_{H1} \otimes 1'_{H1} 1_{H2} \otimes 1'_{H2}$ if only if $h_1\otimes\varepsilon_t(h_2)=1_{H1}h\otimes 1_{H2}$,
			
			\item[(ii)] $\Delta^2(1_H)=1_{H1} \otimes 1_{H2} 1'_{H1} \otimes 1'_{H2}$ if only if $h_1\otimes\varepsilon_s'(h_2)=h1_{H1}\otimes 1_{H2}$, where $\varepsilon_s'(h)=\varepsilon(h1_{H1})1_{H2}$.
		\end{itemize}
	\end{pro}
	
	Consider $ H $ a weak bialgebra. We say that $ H $ is a \textit{weak Hopf algebra} if there is a linear map $ S: H\longrightarrow H$, called \textit{antipode}, such that:
	\begin{enumerate}
		\item [(i)] $h_1S(h_2)=\varepsilon_t(h)$,
		\item [(ii)] $S(h_1)h_2=\varepsilon_s(h)$,
		\item [(iii)] $S(h_1)h_2S(h_3)=S(h),$
	\end{enumerate}
	for all $h \in H$.
	The antipode of a weak Hopf algebra satisfies  $ S (hk) = S (k) S (h) $ 	 and $ S (h) _1 \otimes S (h) _2 = S (h_2) \otimes S (h_1) $, 	for all $h,k \in H$.
	
	\begin{pro}
		Let $H$ be a weak Hopf algebra. Then, the following identities hold for all $h \in H$:
		\begin{eqnarray}
			\varepsilon_t \circ S &=& \varepsilon_t \circ \varepsilon_s = S \circ \varepsilon_s \label{4.34}\\
			\varepsilon_s \circ S &=& \varepsilon_s \circ \varepsilon_t = S \circ \varepsilon_t \label{4.35}\\
			h_1 \otimes S(h_2)h_3 &=& h1_1 \otimes S(1_2) \  \label{4.38}\\
			h_1S(h_2) \otimes h_3 &=& S(1_1) \otimes 1_2h \  \label{4.39}
		\end{eqnarray}
	\end{pro}
	
	It is  straightforward to see that every Hopf algebra is a weak Hopf algebra. Conversely, we have the following result.
	\begin{pro}\label{condicoesdealgebradehopf}
		A weak Hopf algebra is a Hopf algebra if one of the following equivalent conditions is satisfied:
		\begin{itemize}
			\item [(i)] $\Delta(1_H)=1_H \otimes 1_H,$
			\item[(ii)] $\varepsilon(hk) = \varepsilon(h)\varepsilon(k),$
			\item [(iii)]$h_1S(h_2)=\varepsilon(h)1_H,$
			\item [(iv)]$S(h_1)h_2 = \varepsilon(h)1_H,$
			\item [(v)]$H_t=H_s=\Bbbk 1_H,$
		\end{itemize}
		for all $h, k \in H$.	
	\end{pro}

	\begin{ex}[Groupoid Algebra] \label{algebradegrupoide} Let $\mathcal{G}$ be a groupoid such that the cardinality of $\mathcal{G}_{0}$ is finite. Consider $ \Bbbk \mathcal{G}$ the vector space with basis $\lbrace \delta_g \rbrace_{g\in \mathcal{G}}$, indexed by the elements of $\mathcal{G}$. Then, $\Bbbk \mathcal{G}$ is a weak Hopf algebra with the following structures
		\begin{eqnarray*}
			u(1_{\Bbbk})=1_\mathcal{G} = \sum_{e \in \mathcal{G}_{0}} \delta_e, 
			\ \  \ \ \ \ 
			m(\delta_g\otimes \delta_h)=\left\{
			\begin{array}{rl}
				\delta_{gh}, & \text{if $\exists gh$,}\\
				0, & \text{otherwise,}
			\end{array} \right. \\
			\Delta(\delta_g)=\delta_g\otimes \delta_g ,\ \ \ \ \ \ \ \ \ \ \ \ \ \ \ \  \ 
			\varepsilon(\delta_g)=1_{\Bbbk}, \ \ \ \ \ \ \ \ \ \ \ \ \ \ \ \  \  S(\delta_g)=\delta_{g^{-1}}.
		\end{eqnarray*}	
		
		\label{ex_grupoide}
	\end{ex}
	
	Another important example of a weak Hopf algebra is the one based on \cite{Bohmexemplo}*{Ex. 8}.
	It appeared in \cite{Ricardo}, where the group algebra is a (general) finite abelian group instead of a finite cyclic group, as in the previous reference. 
	
	\begin{ex}\label{exemplodogrupo}
		Consider $G$ a finite abelian group with cardinality $|G|$, where $|G|$  is not a multiple of the characteristic of $\Bbbk$. Denote by $\mathcal{H}^G$ the weak Hopf algebra where $\mathcal{H}^G = \Bbbk G$ as algebra, and the coalgebra structure is given by
		\begin{eqnarray*}			
			\Delta(g)=\frac{1}{|G|}\sum_{h \in {G}} gh  \otimes h^{-1}, \ \ \ \ \varepsilon(g)=\left\{
			\begin{array}{rl}
				|G|, & \text{if $g=1_G$,}\\
				0, & \text{otherwise. }
			\end{array} \right.
		\end{eqnarray*}	
		Then, $\mathcal{H}^G$ is a weak Hopf algebra with antipode defined by $S(g)=g$.
		Also, for all $g \in G$, $\varepsilon_s(g)=\varepsilon_t(g)=g$, which implies that $\left(\mathcal{H}^G\right)_s=\left(\mathcal{H}^G\right)_t=\mathcal{H}^G$.
		Furthermore, since $G$ is abelian, $\mathcal{H}^G$ is a commutative and cocommutative weak Hopf algebra.
	\end{ex}

	\begin{ex}\label{exuniaodisjunta}
		The finite disjoint union of Hopf algebras is a weak Hopf algebra.  
		Indeed, if $(H_i, m_i, u_i, \Delta_i, \varepsilon_i, S_i)$ is a weak Hopf algebra, for each $i \in \{ 1, 2,.., n\}$, with $1_{{H}_{i}}=u_i(1_\Bbbk)$ unity of $H_i$, then
		$H={{\mathring{\cup}}_{i=1}}^n H_i$
		is a weak Hopf algebra, where 
		\begin{eqnarray*}
			1_H= \sum_{i=1}^n 1_{H_i}, \ \ \	hk=\left\{
			\begin{array}{rl}
				m_j(h\otimes k), & \text{if $h,k \in H_j$,}\\
				0, & \text{otherwise,}
			\end{array} \right.\\
		\end{eqnarray*}	
		and as usual, for $h \in H_j$, the maps $\Delta, \varepsilon$ and $S$ are defined as $\Delta(h)=\Delta_j(h),$ 
		$\varepsilon(h)=\varepsilon_j(h)$ and $S(h)=S_j(h)$.
	\end{ex}
	
	\begin{ex}[\cite{Kaplansky}] \label{ex_kapla}Let $H$ be a Hopf algebra and $\mathbb{e}$ its unity.
		We consider the set $H'$ as a result of adjoining a new unity $\mathbb{1}$ to $H$, with respect to the multiplication.
		Define
		\begin{align*}
			\Delta(\1) =  (\1-\mathbb{e})\otimes (1-\mathbb{e})+ \mathbb{e}\otimes \mathbb{e},\ \ \ \ \ \ 
			\varepsilon(\1) =  2,\ \ \ \ \ \
			S(\1) = & \1.
		\end{align*}
		Then $H'$ becomes a weak Hopf algebra.
	\end{ex}
	
	\subsection{(Co)Action of a weak bialgebra over an algebra}

	\begin{defi}[\cite{Caenepeel_}]
		Let $H$ be a weak bialgebra and $A$ an algebra. If there exists a linear map
		\begin{eqnarray*}
			\cdot: H \otimes A & \longrightarrow & A\\
			h \otimes a & \longmapsto & h\cdot a
		\end{eqnarray*}
		such that, for all $h,k\in H$ and $a,b\in A$, the following conditions are satisfied:
		\begin{enumerate}
			\item [(MA1)] $1_H \cdot a=a$,
			\item [(MA2)] $h \cdot ab=(h_1 \cdot a)(h_2 \cdot b)$,
			\item [(MA3)] $h\cdot(k \cdot a)=hk\cdot a$,
			\item[(MA4)] $h \cdot 1_A = \varepsilon_t(h) \cdot 1_A$,
		\end{enumerate}
		then we say that $A$ is a \textit{(left) $H$-module algebra}.
		The map $\cdot$ is called a \emph{(left) action of $H$ on $A$}.
	\end{defi}
	
	\begin{rem}[\cite{Felipe}]
		If $H$ is a weak Hopf algebra, then conditions (MA1)-(MA3) imply (MA4). 
	\end{rem}
	
	\begin{ex}\label{exquedacertoacao}
		$\mathcal{H}^G$ is an $\mathcal{H}^G$-module algebra via its multiplication.
	\end{ex}
	
	There exists a particular case that we are interested: actions of a weak Hopf algebra $H$ over an algebra $A$ via scalars, as presented in the next example.
	\begin{ex}[\cite{Felipe}]\label{ex_lambda_geral}
		Consider $H$ a weak Hopf algebra, $A$ an algebra and a linear map $\lambda : H \longrightarrow \Bbbk$, such that $h \cdot a = \lambda(h)a$. The linear map $\lambda$ defines an action of $H$ on $A$ if and only if it satisfies
		\begin{eqnarray*}
			\lambda(1_H) & = & 1_\Bbbk, \\
			\lambda (h) & = & \lambda (h_1) \lambda (h_2), \\
			\lambda (h)\lambda (k) & = & \lambda (hk),
		\end{eqnarray*}
		for all $h, k \in H$.  
		
		In particular, $\lambda$ defines an action of $H$ on the base field $\Bbbk$ via $h \cdot 1_\Bbbk = \lambda(h)$.
	\end{ex}

	\begin{ex}
		Consider $n$ groups $G_1,...,G_n$ and the groupoid given by the disjoint union of them, that is, $\mathcal{G} = \mathring{\cup}_{i=1}^{n} G_i.$
		Then, $\lambda$ is an action of the groupoid algebra $\Bbbk\mathcal{G}$ on its base field, if and only if, $\lambda = \lambda_\ell$, where 
		\begin{eqnarray*}
			\lambda_\ell(\delta_g)=\left\{
			\begin{array}{rl}
				1, & \text{if $g \in G_\ell$,}\\
				0, & \text{otherwise.}
			\end{array} \right.
		\end{eqnarray*}
	\end{ex}
	
	Although the above example is elementary, it can be proved that a connected groupoid algebra acts on the base field via the functional $\lambda$ if and only if the groupoid is a group.
	Moreover, in that case, the map $\lambda$ is the counit map.
	Hence, the previous example is the only possible for groupoids.
	
	Since a group algebra is a Hopf algebra, the next example generalizes the previous one.
	\begin{ex}
		Let $H = \mathring{\cup}_{i=1}^{n} H_i$ be the weak Hopf algebra given as the finite disjoint union of Hopf algebras (see Example \ref{exuniaodisjunta}).
		Then, $\lambda$ is an action of the 
		$H$ on its base field, if and only if, $\lambda = \varepsilon_\ell$, where $\varepsilon_\ell$ is the counit of $H _\ell$, for some $\ell \in \{1, ..., n\}$.
	\end{ex}
	
	\begin{ex} Let $G$ be a finite group and consider the weak Hopf algebra $\mathcal{H}^G$ (see Example \ref{exemplodogrupo}).
		Since, the compatibility $\lambda(g)=\lambda(g_1)\lambda(g_2)$ holds for all morphism of algebras $\lambda: \mathcal{H}^G \longrightarrow \Bbbk$, we get that $\lambda$ is an action (see Example \ref{ex_lambda_geral}) if and only if $\lambda$ is a character.
		
		In particular, considering $C_n = \langle g \ | \ g^n =1 \rangle$ the cyclic group of order $n$, every action $\lambda$ of $\mathcal{H}^{C_n}$ on its base field is determined by $\lambda = \lambda_q$, where $\lambda_q (g^i) =q^i$, for $q$ an $nth$ root of unity.
	\end{ex}

	\begin{ex}\label{Ex_acao_esthopf}
		Let $H$ be a Hopf algebra, $H'$ the weak Hopf algebra of Example \ref{ex_kapla} and $A$ an algebra.
		If $A$ is an $H$-module algebra, then there is a unique structure of $H'$-module algebra on $A$ that extends it, and it is given in the usual way.
	\end{ex}
	
	The following lemma is \cite{Felipe}*{Lem. 2.18} and it was proved in the context of a \emph{partial $H$-module algebra} (for the definition, see \cite{Felipe}).
	Since every $H$-module algebra is a partial $H$-module algebra in the usual way, we will be able to use it.
	\begin{lem}\label{lema_2_1_9_Glauber}
		Let $A$ be a partial $H$-module algebra, $a, b \in A$ and $h \in H$. Then:
		\begin{itemize}
			\item[(i)] if $h \in H_t$, then $(h \cdot a)b = h \cdot ab$,
			\item[(ii)] if $h \in H_s$, then $a(h \cdot b) = h \cdot ab$.
		\end{itemize}
	\end{lem}
	
	The following definition is dual to the $H$-module algebra structure:
	
	\begin{defi}[\cite{Wang}]
		Consider $H$ a weak bialgebra and $C$ a coalgebra.
		We call $C$ a \textit{(right) $H$-comodule coalgebra}  when  there exits a linear map 
		\begin{eqnarray*}
			\rho: C & \longrightarrow & C\otimes H\\
			c & \longmapsto & c^{0} \otimes c^{1} 
		\end{eqnarray*}
		such that, for all $c\in C$, the following conditions hold:
		\begin{enumerate}
			\item [(CC1)] $\varepsilon(c^{1})c^{0}=c$,
			\item [(CC2)] ${c^{0}}_{1} \otimes {c^{0}}_{2}\otimes c^{1}  =  {c_{1}}^{0} \otimes {c_{2}}^{0}\otimes {c_{1}}^{1}{c_{2}}^{1}$,
			\item [(CC3)]$  c^{00}\otimes c^{01} \otimes c^{1}=  c^{0}\otimes {c^{1}}_{1} \otimes {c^{1}}_{2}$,
			\item[(CC4)] $c^{1}\varepsilon(c^{0})=\varepsilon_s(c^{1})\varepsilon(c^{0})$.
		\end{enumerate}
		The map $\rho$ is called a \emph{(right) coaction of $H$ on $C$}.
	\end{defi}
	
	\begin{rem} \cite{weak_smash_coproduct}*{Prop. 3.1}
		Let $H$ be a weak Hopf algebra. If the map $\rho: C \longrightarrow C \otimes H$ satisfies conditions (CC1)-(CC3), then condition (CC4) is also satisfied.
	\end{rem}

	\begin{ex}\label{exquedacertocoacao}
		Consider the weak Hopf algebra $\mathcal{H}^G$ (see Example \ref{exemplodogrupo}). Then $\mathcal{H}^G$ is an $\mathcal{H}^G$-comodule coalgebra via its comultiplication map $\Delta$.
	\end{ex}
	
	Now, for a coalgebra $C$ and an element $z \in H$, we denote by $\rho_z^C$ the linear map $\rho_z^C: C \longrightarrow C \otimes H$ given by $\rho_z^C(c) = c \otimes z$, for all $c \in C$. As in the case of actions, we are particularly interested in coactions of a weak Hopf algebra $H$ over a coalgebra $C$ through a coaction of $H$ on the base field $\Bbbk$ as follows.
	
	\begin{ex}[\cite{weak_smash_coproduct}]\label{ex_z_geral}
		Consider $H$ a weak Hopf algebra and $C$ a coalgebra. Then, $C$ is a $H$-comodule coalgebra via $\rho_z^C$ if and only if $z \in H$ satisfies
		\begin{eqnarray*}
			\varepsilon (z) & = & 1_\Bbbk, \\
			z & = & z^2, \\
			\Delta(z) & = & z \otimes z.
		\end{eqnarray*}
		In particular, the element $z \in H$ defines a coaction of $H$ on $C$ via $\rho^C_z$ if and only if $\rho^\Bbbk_z$ is a coaction of $H$ on the base field $\Bbbk$.
		
	\end{ex}
	
	\begin{ex}
		Let $\mathcal{G} = \mathring{\cup}_{i=1}^{n} G_i$ be the disjoint union of the groups $G_1,...,G_n$.
		Then, $z$ is a coaction of the groupoid algebra $\Bbbk\mathcal{G}$ on its base field, if and only if, $z=\delta_{e_\ell}$, where $e_\ell$ is the identity element of $G_\ell$, for some $\ell \in \{1, ..., n\}$.	
	\end{ex}
	
	\begin{ex}
		Let $H = \mathring{\cup}_{i=1}^{n} H_i$ be the disjoint union of Hopf algebras $H_1,...,H_n$.
		Then, $z$ is a coaction of the $H$ on its base field, if and only if, $z=1_\ell$, where $1_\ell$ is the unity of $H _\ell$, for some $\ell \in \{1, ..., n\}$.
	\end{ex}

	\begin{ex}
		Consider the weak Hopf algebra $\mathcal{H}^{G}$ (see Example \ref{exemplodogrupo}).
		Every coaction of $\mathcal{H}^{G}$ on its base field is determined by $z = \sum_{g \in G} \alpha_g g$, where $(\alpha_g)_{g \in G} \in \Bbbk^{|G|}$ satisfies $\alpha_e=1/|G|$ and $\alpha_g \alpha_h = \frac{1}{|G|} \alpha_{gh},$ for all $g, h \in G$.
		Moreover, it is possible to conclude that such $z \in \mathcal{H}^G$ thus determined exists if and only if $\beta: G \longrightarrow \Bbbk$ is a character of $G$, where $\beta(g)=|G|\alpha_g$.
		
		In particular, considering $C_n$ the cyclic group of order $n$, every coaction $z$ of $\mathcal{H}^{C_n}$ on its base field is determined by $z_q = \frac{1}{n} \sum_{i=1}^n q^i g^i$, where $q$ is an $nth$ root of unity
	\end{ex}
	
	\begin{ex}\label{Ex_coacao_esthopf}
		Let $H$ be a Hopf algebra, $H'$ the weak Hopf algebra of Example \ref{ex_kapla} and $C$ a coalgebra.
		If $C$ is an $H$-comodule algebra, then there is a unique structure of $H'$-comodule algebra on $C$ that extends it, and it is given in the usual way.
	\end{ex}
	
	\section{Weak  matched pair}\label{sec_weak_matched_pair}
	
	In this section we introduce the definition of a matched pair of weak bialgebras.
	It generalizes the classical definitions of matched pairs of bialgebras and Hopf algebras \cites{Majid, Takeuchi}, and also was inspired in the definition of a \emph{partial matched pair} of Hopf algebras \cite{matchedpair}.
	
	\begin{defi}[\cite{Majid}]\label{matched pair} A pair of bialgebra $(H,A)$ is called a \emph{matched pair} if
		\begin{itemize}
			\item[(i)] $A$ is a left $H$-module algebra with action given by $h\triangleright x$,
			\item[(ii)] $H$ is a right $A$-comodule coalgebra with structure $\rho:H\longrightarrow H\otimes A$, denoted by $\rho(h)=h^0\otimes h^1$,
			\item[(iii)] $\rho(hg)=\rho(h_1)(g^0\otimes (h_2\triangleright g^1))$,
			\item[(iv)] $\Delta(h\triangleright x)=({h_1}^0\triangleright x_1)\otimes ({h_1}^1(h_2\triangleright x_2)) $,
			\item[(v)] $\varepsilon(h\triangleright x)=\varepsilon(h)\varepsilon(x)$,
			\item[(vi)] $\rho(1_H)=1_H\otimes 1_A$,
		\end{itemize}
		for all $h,g\in H$ and $x\in A$. The matched pair $(H,A)$ is called \emph{abelian} if $H$ is cocommutative
		and $A$ is commutative.
	\end{defi}
	
	Takeuchi, in \cite{Takeuchi}, showed that when $H$ and $A$ are Hopf algebras, the items (v) and (vi) are automatically satisfied.
	Now, the following definition extends the previous one to a weak bialgebra context.
	
	\begin{defi}\label{novo weak matched pair} A pair of weak bialgebras $(H,A)$ is called a \textit{weak matched pair} if
		\begin{enumerate}
			\item[(i)] $A$ is a left $H$-module algebra via $\cdot$,
			\item[(ii)] $H$ is a right $A$-comodule coalgebra via $\rho$, which is denoted by $\rho(h)= h^0 \otimes h^1$,
			\item[(iii)] $(\tau_{A,H}\otimes m_A)(id\otimes \tau_{A,H}\otimes id)(\Delta_A(h_1\cdot x)\otimes \rho(h_2g))$\\ \vspace{-0.4cm}
			\begin{flushright}
				=\ ${h_3}^0 g^0\otimes ({h_1}^0\cdot x_1)\otimes {h_1}^1(h_2\cdot x_2){h_3}^1(h_4\cdot g^1),$
			\end{flushright}
			\item[(iv)] $\varepsilon_s(h\cdot x)=\varepsilon_s(h\cdot 1_A)\varepsilon_s(x)$,
			\item [(v)] $\rho(1_{H1})\otimes 1_{H2}=1_{H1}\otimes \varepsilon({1_{H1}'}^0){1_{H1}'}^1\otimes 1_{H2} 1_{H2}'$,	
		\end{enumerate}
		for all $h,g\in H$ and $x\in A$, where $\tau_{A,H}: A\otimes H\longrightarrow H\otimes A$ is the twist map.
		Also, we call \emph{abelian weak matched pair} when $H$ is cocommutative and $A$ is commutative.
	\end{defi}

	\begin{rem}\label{troca}
		Item (iii) of the above definition can be rewritten as follows:
		\begin{equation}\label{igualdade_iii}
			\begin{split}
				& (h_2 g)^0\otimes (h_1\cdot x)_1\otimes (h_1\cdot x)_2 (h_2 g)^1 \\ 
				& = {h_3}^0 g^0
				\otimes ({h_1}^0\cdot x_1)\otimes  {h_1}^1(h_2\cdot x_2){h_3}^1(h_4\cdot g^1),
			\end{split}	
		\end{equation}
		for all $h,g\in H$ and $x\in A$.
	\end{rem}
	
	\begin{rem}\label{vart_acao}
		If $A$ is a left $H$-module algebra via $\cdot$ such that the item \ref{novo weak matched pair}(iv) is satisfied with $A$ commutative, then $\varepsilon_t(h\cdot x)=\varepsilon_t(h\cdot 1_A)\varepsilon_t(x)$, for all $x\in A$ and $h\in H$. Indeed, 
		\begin{equation*}
			\begin{split}
				\varepsilon_t(h\cdot x) & = \varepsilon(1_{A1}(h\cdot x))1_{A2}\\
				&\stackrel{(\ref{4.6})}{=}  \varepsilon(1_{A1}\varepsilon_s(h\cdot x))1_{A2}\\
				&  \stackrel{\ref{novo weak matched pair}(iv)}{=} \varepsilon(1_{A1}\varepsilon_s(h\cdot 1_A)\varepsilon_s(x))1_{A2}\\
				& =  \varepsilon_t\left((h\cdot 1_A)x\right)\\
				& =  \varepsilon_t\left(h\cdot 1_A\right) \varepsilon_t\left(x\right), 
			\end{split}
		\end{equation*}
		where, in the last equality, we use that $\varepsilon_t$ is a multiplicative map since $A$ is commutative.
	\end{rem}

	\begin{pro}
		Let $H$ and $A$ be bialgebras. If $(H,A)$ is a matched pair, then $(H,A)$ is a weak matched pair.
	\end{pro}
	
	\begin{proof}
		This follows directly from Definitions \ref{matched pair} and \ref{novo weak matched pair}, using that $\varepsilon_s(x)=\varepsilon(x)1_A$, for all $x\in A$ and $\Delta(1_H)=1_H\otimes 1_H$. Note that, items (iii) and (iv) of Definition \ref{matched pair} imply the item (iii) of Definition \ref{novo weak matched pair}.
	\end{proof}
	
	Next, we present examples of weak matched pairs $(H,A)$.
	For that, it is routine computation to verify   items (i)-(v) of Definition \ref{novo weak matched pair}.
	In particular, for these examples, we use the structure of (left) $H$-module algebra on $A$ and (right) $A$-comodule coalgebra on $H$ 
	as presented in the Section \ref{sec_weak_preliminares}.
	Thus, we only need to check items (iii), (iv) and (v), where usually one uses Remark \ref{troca} for the first.
	
	\begin{ex} Consider the weak Hopf algebra $\mathcal{H}^G$ and the action and coaction on itself as in Examples  \ref{exquedacertoacao} and \ref{exquedacertocoacao}, respectively.
		By Remark \ref{troca}, both sides of item (iii) means 
		\begin{equation*}
			\frac{1}{|G|^3} \sum_{u, v, t}  g t u^{-1} \otimes h x u v \otimes v^{-1} t^{-1}.
		\end{equation*}
		Notice that to get it in the right hand of equality \eqref{igualdade_iii}, one uses several times that $\sum_{a,b,c \in G} a \otimes bc \otimes c^{-1}b^{-1} = |G| \sum_{a,b \in G} a \otimes b \otimes b^{-1}$.
		
		Item (iv) is clear since $\mathcal{H}^G$ is commutative, and so it implies that $\varepsilon_s$ is a multiplicative map.
		At last, for item (v), 
		\begin{equation*}
			\rho(1_{H1})\otimes 1_{H2}= \frac{1}{|G|^2}\sum_{g,h \in G} gh \otimes h^{-1}\otimes g^{-1}= 1_{H1}\otimes \varepsilon({1_{H1}'}^0){1_{H1}'}^1\otimes 1_{H2} 1_{H2}'.
		\end{equation*}
		Hence, \label{exemplo_par_combinado_HG}
		$(\mathcal{H}^G, \mathcal{H}^G)$ is an abelian weak matched pair.
	\end{ex}
	
	\begin{ex}\label{exemplo_par_combinado_lambda_z}
		Assume that $H$ acts on $A$ via $\lambda \in H^*$ (see Example \ref{ex_lambda_geral}) and $A$ coacts on $H$ via $z \in A$ (see Example \ref{ex_z_geral}).
		Then, $(H,A)$ is a weak matched pair if and only if $\lambda(h_1) h_2 \lambda(h_3)=\lambda(h_1) h_2$ and $zxz=xz$, for all $h \in H$ and $x \in A$.
		Indeed, items (iv) and (v) are clear.
		Then, for item (iii), on the one hand
		\begin{equation*}
			(h_2 g)^0\otimes (h_1\cdot x)_1\otimes (h_1\cdot x)_2 (h_2 g)^1 = \lambda(h_1)  h_2 g \otimes x_1 \otimes x_2 z,
		\end{equation*}
		and, on the other hand,
		\begin{equation*}
			\begin{split}
				&{h_3}^0 g^0 \otimes ({h_1}^0\cdot x_1)\otimes {h_1}^1(h_2\cdot x_2){h_3}^1(h_4\cdot g^1) \\
				&= {h_3}^0 g^0 \otimes \lambda({h_1}^0) x_1 \otimes {h_1}^1 \lambda(h_2) x_2 {h_3}^1 \lambda(h_4) g^1\\
				&= h_3 g \otimes \lambda(h_1) x_1 \otimes z \lambda(h_2) x_2 z \lambda(h_4) z \\
				&= \lambda(h_1) \lambda(h_2) h_3 \lambda(h_4) g \otimes x_1 \otimes z x_2 z^2 \\
				&= \lambda(h_1) h_2 \lambda(h_3) g \otimes x_1 \otimes z x_2 z,
			\end{split}
		\end{equation*}
		where we use that $\lambda(k_1)\lambda(k_2)=\lambda(k)$ and $z^2=z$ in the last equality.
		
		Hence, item (iii) holds if and only if $\lambda(h_1) h_2 \lambda(h_3)=\lambda(h_1) h_2$ and $zxz=xz$, for all $h \in H$ and $x \in A$.
		
		In particular, if $H$ is cocommutative and $A$ is commutative, then these conditions hold.
		In that case, $(H,A)$ is an abelian matched pair.
	\end{ex}
	
	\begin{ex}\label{ex_kap2}
		Let $H$ and $A$ be Hopf algebras, and $H'$ and $A'$ the weak Hopf algebras as in Example \ref{ex_kapla}. If $(H,A)$ is a matched pair, then $(H',A')$ is a weak matched pair.
		Indeed, just define
		$\1_{H'}\cdot x=x$, for all $x\in A'$, $h\cdot\1_{A'}=\varepsilon(h)\1_{A'},$ for all $h\in H$ and $$\rho(\1_{H'})=(\1_{H'}-\mathbb{e}_{H'})\otimes (\1_{A'}- \mathbb{e}_{A'})+ \mathbb{e}_{H'}\otimes \mathbb{e}_{A'}.$$ 
		Then, it is only routine computation to verify that all items of Definition \ref{novo weak matched pair} hold.
	\end{ex}
	
	\medskip
	
	Let $B$ and $L$ be weak bialgebras.
	Similar to Definition \ref{novo weak matched pair}, if $B$ is a right $L$-module algebra and $L$ is a left $B$-comodule coalgebra, one can define a weak matched pair $(B,L)$.
	
	Consider $(H,A)$ a weak matched pair (as in Definition \ref{novo weak matched pair}).
	Assume that $H$ and $A$ are finite dimensional weak bialgebras.
	Since $H$ is a right $A$-comodule coalgebra, then $H$ is a left $A^*$-module coalgebra via $f \triangleright h = f(h^1)h^0$ and, furthermore, $H^*$ is a right $A^*$-module algebra via $(\phi \leftharpoonup f)(h)=f(h^1)\phi(h^0)$, for all $f \in A^*$, $\phi\in A^*$ and $h \in H$. Also, since $A$ is a left $H$-module algebra, then $A^*$ is a right $H^{**}$-module coalgebra via $(f \leftharpoondown \widehat{h})(x)=\widehat{x}(f \leftharpoondown \widehat{h})=f(h \rightharpoonup x)$, for all $x \in A$, $f\in A^*$ and $\widehat{h}=h^{**} \in H^{**}$. Furthermore, $A^*$ is a left $H^*$-comodule coalgebra via $\rho(f)=\sum_{i=1}^{n} {h_i}^* \otimes f \leftharpoondown \widehat{h_i}$, where $\{{h_i}^*\}_{i=1}^{n}$ is the dual basis in $H^*$ (see, \cites{weak_smash_coproduct, Caenepeel_}).
	
	Thus, we have  the following result.
	\begin{pro}
		Let $H$ and $A$ be finite dimensional weak bialgebras. If $(H,A)$ is a weak matched pair, where $A$ is a left $H$-module algebra and $H$ is a right $A$-comodule coalgebra, then $(A^{\ast},H^{\ast})$ is a weak matched pair, where $H^*$ is a right $A^*$-module algebra and $A^*$ is a left $H^*$-comodule coalgebra.
	\end{pro}
	
	\section{The weak Hopf algebra $A \underline{\overline{\#}} H$}
	
	In this section, we analyze under what conditions it is possible to construct a weak bialgebra/Hopf algebra involving the usual smash product and coproduct.
	
	\subsection{The structure of weak bialgebra}
	
	Consider $H$ and $A$ two weak bialgebras.
	First, if $A$ is a $H$-module algebra, then we can define $A \# H$ as $A \otimes H$ with the product
	$$(x\# h)(y\# g) = x(h_1 \cdot y) \# h_2g,$$
	for all $x,y\in A$ and $h,g\in H$.
	It is straightforward to check the associativity of this product.
	However, the natural candidate for unity of this product, the element $1_A \# 1_H$, is neither right nor left unity in $A \# H$. 
	Then, we need to consider a smaller space of it.
	Since $(x\# h)(1_A \# 1_H)=(1_A \# 1_H)(x\# h)(1_A \# 1_H)$, for all $(x\# h) \in A \# H$, we denote
	$$A \underline{\#} H=  (A \# H)(1_A \# 1_H)=\{x(h_1 \cdot 1_A) \# h_2 \ | \ x \in A, h \in H \}$$ 
	and $x \underline{\#} h = (x \# h)(1_A \# 1_H) = x(h_1 \cdot 1_A) \# h_2 \in A \underline{\#} H$. The next result ensure that $A \underline{\#} H$ is an algebra, with unity $1_A \underline{\#} 1_H$.
	
	\begin{pro} Let $H$ and $A$ be weak bialgebras, and $A$ an $H$-module algebra. Then, $A \underline{\#} H$ is an algebra.
	\end{pro}
	
	\begin{proof}
		First we show that $(x \underline{\#} h)(y \underline{\#} g) \in A \underline{\#} H$.
		Indeed:
		\begin{eqnarray*}
			(x \underline{\#} h)(y \underline{\#} g) &=&  (x({h}_{1} \cdot 1_A) \# {h}_{2}) (y({g}_{1} \cdot 1_A) \# {g}_{2}) \\ 
			&=&  (x({h}_{1} \cdot 1_A)({h}_{2} \cdot y(g_1 \cdot 1_A)) \# h_3{g}_{2}) \\
			&=&  (x({h}_{1} \cdot y)({h}_{2} g_1 \cdot 1_A) \# h_3{g}_{2}) \\
			&=&  (x({h}_{1} \cdot y) \# h_2g) (1_A \# 1_H)  \\
			&=& x({h}_{1} \cdot y) \underline{\#} h_2g.
		\end{eqnarray*}
		
		Now, for all $(x\# h) \in A \# H$,  $(1_A \# 1_H) (1_A \# 1_H)(x\# h)=(1_A \# 1_H)(x\# h) $ and  $(x\# h)(1_A \# 1_H) (1_A \# 1_H)= (x\# h)(1_A \# 1_H) $.
		Thus, $1_A \underline{\#} 1_H$ is the unity in $A \underline{\#} H$:
		\begin{eqnarray*}
			(1_A \underline{\#} 1_H) (x \underline{\#} h) &=&  (1_A \# 1_H)(1_A \# 1_H) (x \# h) (1_A \# 1_H)\\
			&=&  (1_A \# 1_H)(x \# h) (1_A \# 1_H)\\
			&=& x \underline{\#} h,
		\end{eqnarray*}
		and also
		\begin{eqnarray*}
			(x \underline{\#} h) (1_A \underline{\#} 1_H) &=&  (x \# h)(1_A \# 1_H)(1_A \# 1_H)(1_A \# 1_H)\\
			&=& (x \# h)(1_A \# 1_H)\\
			&=& x \underline{\#} h.
		\end{eqnarray*}
		
		Therefore, $A \underline{\#} H$ is an  algebra. 
	\end{proof}
	
	\begin{rem} If $H$ is a weak bialgebra cocommutative, then $(x \# h)(1_A \# 1_H) = (1_A \# 1_H)(x \# h)$, for all $(x \# h) \in A \# H$. Indeed,
		\begin{eqnarray*}
			(x \# h)(1_A \# 1_H) &=&  x(h_1 \cdot 1_A ) \# h_2\\
			&=& 1_H \cdot ( x(h_1 \cdot 1_A )) \# h_2\\
			&\stackrel{(\ref{4.7})}{=}& ({1_H}_1 \cdot x)  ( \varepsilon_t({1_H}_2)h_1 \cdot 1_A ) \# h_2\\
			&=& ({1_H}_1 \cdot x) ( \varepsilon_t( \varepsilon_t({1_H}_2)h_1) \cdot 1_A ) \# h_2\\
			&=& ({1_H}_1 \cdot x)  (\varepsilon_t({1_H}_2) \varepsilon_t(h_1) \cdot 1_A ) \# h_2\\
			&=& ({1_H}_1 \cdot x) ({1_H}_2 \varepsilon_t(h_1) \cdot 1_A ) \# h_2\\
			&\stackrel{H \ coco}{=}& ({1_H}_1 \cdot x) ({1_H}_2 \varepsilon_t(h_2) \cdot 1_A ) \# h_1\\
			&\stackrel{(\ref{4.12})}{=}& ({1_H}_1 \cdot x)  ({1_H}_2 {1_H'}_2  \cdot 1_A ) \# {1_H'}_1h\\
			&\stackrel{H \ coco}{=}& ({1_H}_1 \cdot x)  ({1_H}_2 {1_H'}_1  \cdot 1_A ) \# {1_H'}_2h\\
			&=& ({1_H}_1 \cdot x)  ({1_H}_2   \cdot 1_A ) \# {1_H}_3h\\
			&=&  ({1_H}_1 \cdot x ) \# {1_H}_2h\\ 
			&=& (1_A \# 1_H)(x \# h).
		\end{eqnarray*}
	\end{rem}
	
	\medskip
	
	Now, if $H$ an $A$-comodule coalgebra, we also define $A \# H$ as $A\otimes H$ with the coproduct 
	\begin{eqnarray*}
		\Delta: A \# H &\longrightarrow & A \# H \otimes A \# H\\
		x \# h & \longmapsto & \Delta(x \# h) = x_1 \# {h_1}^{0} \otimes x_2 {h_1}^{1} \# h_2
	\end{eqnarray*}
	and $\varepsilon: A \# H \longrightarrow  \Bbbk$ given by $\varepsilon(x \# h) = \varepsilon(x) \varepsilon(h),$ for all $x\in A$ and $h\in H$.
	It is straightforward to check that $\Delta$ is coassociative, however $\varepsilon$ is neither right nor left counit.
	Therefore, we denote $\overline{x \# h} = (id \otimes \varepsilon)\Delta(x \# h) = x_1 \varepsilon(x_2 h^1) \# h^0$ and consider the subspace
	$$\overline{A \# H} = (id \otimes \varepsilon) \Delta (A \# H ) = \{ x_1 \varepsilon(x_2 h^1) \# h^0 \ | \ x \in A, h \in H \}.$$
	
	\begin{pro}\label{delta}
		Let $H$ and $A$ be weak bialgebras, and $H$ an $A$-comodule coalgebra. Then, 	$\overline{A \# H}$ is a coalgebra with coproduct $\Delta$ and counit $\varepsilon$.
	\end{pro}
	
	\begin{proof}
		Initially, we show that $\Delta( \overline{x \# h} ) \in \overline{A \# H} \otimes \overline{A \# H}$. Indeed:
		\begin{eqnarray*}
			\Delta( \overline{x \# h} ) &=& \Delta (x_1 \varepsilon(x_2 h^1) \# h^0)\\    
			&=& x_1 \varepsilon(x_3 h^1) \# {{h^0}_1}^{0} \otimes x_2 {{h^0}_1}^{1} \# {{h^0}_2}\\    
			&=& x_1 \varepsilon(x_3 {h_1}^1{h_2}^1) \# {{h}_1}^{00} \otimes x_2 {{h}_1}^{01} \# {{h_2}^{0}}\\    
			&=& x_1 \varepsilon(x_3 {{h_1}^1}_2{h_2}^1) \# {{h}_1}^{0} \otimes x_2 {{h_1}^1}_1 \# {{h_2}^{0}}\\    
			&=& x_1 \varepsilon(x_3 {{h_1}^1}_2) \varepsilon(x_4{{h_1}^1}_3{h_2}^1) \# {{h}_1}^{0} \otimes x_2 {{h_1}^1}_1 \# {{h_2}^{0}}\\    
			&=& x_1 \varepsilon(x_2 {{h_1}^1}_1)  \# {{h}_1}^{0} \otimes x_3 {{h_1}^1}_2 \varepsilon(x_4{{h_1}^1}_3{h_2}^1) \# {{h_2}^{0}}\\    
			&=& \overline{x_1 \# {h_1}^{0}} \otimes \overline{x_2 {h_1}^{1}\# {h_2}}.
		\end{eqnarray*}
		
		Now, we verify that $\varepsilon$ is a counit:
		\begin{eqnarray*}
			(\varepsilon \otimes id) \Delta ( \overline{x \# h} ) &=& \varepsilon  (\overline{x_1 \# {h_1}^{0}})  \overline{x_2 {h_1}^{1}\# {h_2}}\\
			&=& \varepsilon  (x_1 {h_1}^{01}) \varepsilon ({h_1}^{00})  \overline{x_2 {h_1}^{1}\# {h_2}}\\
			&=& \varepsilon  (x_1 {{h_1}^{1}}_1) \varepsilon ({h_1}^{0})  \overline{x_2 {{h_1}^{1}}_2\# {h_2}}\\
			&=& \varepsilon ({h_1}^{0})  \overline{x {{h_1}^{1}}\# {h_2}}\\
			&=& \varepsilon ({h_1}^{0})  x_1 {{h_1}^{1}}_1 \varepsilon(x_2 {{h_1}^{1}}_2 {{h_2}^{1}})\# {{h_2}^{0}}\\
			&=& \varepsilon ({h_1}^{0})  x_1 \varepsilon_s({{h_1}^{1}})_1 \varepsilon(x_2 \varepsilon_s({{h_1}^{1}})_2 {{h_2}^{1}})\# {{h_2}^{0}}\\
			&=&   x_1 {1_A}_1 \varepsilon(x_2 \varepsilon_s({{h_1}^{1}}) {1_A}_2 {{h_2}^{1}})\# \varepsilon ({h_1}^{0}) {{h_2}^{0}}\\
			&=&   x_1 {1_A}_1 \varepsilon(x_2 {1_A}_2  \varepsilon_s({{h_1}^{1}}) {{h_2}^{1}})\# \varepsilon ({h_1}^{0}) {{h_2}^{0}}\\
			&=&   x_1 \varepsilon(x_2 {{h_1}^{1}} {{h_2}^{1}})\# \varepsilon ({h_1}^{0}) {{h_2}^{0}}\\
			&=&   x_1 \varepsilon(x_2 {{h_1}^{1}})\# \varepsilon ({{h}^{0}}_1) {{h}^{0}}_2\\
			&=& \overline{x \# h},
		\end{eqnarray*}
		and also
		\begin{eqnarray*}
			(id \otimes \varepsilon) \Delta ( \overline{x \# h} ) &=& \overline{x_1 \# {h_1}^{0}}  \varepsilon  ( \overline{x_2 {h_1}^{1}\# {h_2}})\\    
			&=& \overline{x_1 \# {h_1}^{0}}  \varepsilon  ( x_2 {h_1}^{1} {h_2}^{1}) \varepsilon({h_2}^0)\\  
			&=& \overline{x_1 \# {h^0}_{1}}  \varepsilon  ( x_2 {h^1}) \varepsilon({h^0}_2)\\    
			&=& \overline{x_1 \# {h^0}}  \varepsilon  ( x_2 {h^1}) \\    
			&=& x_1 \# {h}^{00}  \varepsilon  ( x_2 {h^{01}})  \varepsilon  ( x_3 {h^{1}}) \\    
			&=& x_1 \# {h}^{0}  \varepsilon  ( x_2 {h^{1}}_1)  \varepsilon  ( x_3 {h^{1}}_2) \\    
			&=&\overline{x \# h}.   
		\end{eqnarray*}
		Therefore, $(\overline{A \# H}, \Delta, \varepsilon)$ is a coalgebra.
	\end{proof}
	
	We observe that such a coalgebra is a projection, since $\overline{\overline{x \# h}} = \overline{x \# h}$, for all $x \# h \in A \# H$, and so $\overline{\overline{A \# H}} = \overline{A \# H}$.

	\begin{pro}\label{delta_mult}
		Let $(H,A)$ be an abelian weak matched pair.
		Then, $\Delta$ is multiplicative in $A \# H$.
	\end{pro}
	
	\begin{proof}
		For all $(x \# h), (y \# g)$ in $A \# H$,
		\begin{eqnarray*}
			\Delta ((x \# h) (y \# g)) &=& \Delta (x (h_1 \cdot y) \# h_2g)\\
			&=& x_1 {(h_1 \cdot y)}_{1} \# {(h_2g_1)}^{0} \otimes x_2 {(h_1 \cdot y)}_{2}{(h_2g_1)}^{1} \# h_3g_2\\
			&\stackrel{\ref{troca}}{=}& x_1 ({h_1}^{0} \cdot y_1) \# {h_3}^{0}{g_1}^{0} \otimes x_2 {h_1}^{1}(h_2 \cdot y_2){h_3}^{1}(h_4 \cdot {g_1}^{1}) \# h_5g_2\\
			&=& x_1 ({h_1}^{0} \cdot y_1) \# {h_2}^{0}{g_1}^{0} \otimes x_2 {h_1}^{1}{h_2}^{1}(h_3 \cdot y_2)(h_4 \cdot {g_1}^{1}) \# h_5g_2\\
			&=& x_1 ({{h_1}^{0}}_1 \cdot y_1) \# {{h_1}^{0}}_2{g_1}^{0} \otimes x_2 {h_1}^{1}(h_2 \cdot y_2)(h_3 \cdot {g_1}^{1}) \# h_4g_2\\
			&=& \Delta (x \# h) \Delta (y \# g).
		\end{eqnarray*}
	\end{proof}
	
	\medskip
	
	Now, in order to obtain a structure of weak bialgebra in a subspace of $A \otimes H$, the abelian weak matched pair $(H,A)$ needs to satisfy some compatible conditions.
	
	\begin{defi}\label{weak matched pair} An abelian weak matched pair $(H,A)$ is called a \textit{compatible weak matched pair} if, for all $h,g\in H$,
		\begin{enumerate}
			\item[(i)] $h\cdot 1_A=h_1\cdot 1_{A1}\varepsilon(h_2\cdot 1_{A2})$,
			\item[(ii)] $h\cdot 1_{A1}\otimes g\cdot 1_{A2}=(h\cdot 1_A)1_{A1}\otimes (g\cdot 1_A)1_{A2}$.
		\end{enumerate}
	\end{defi}
	
	\begin{lem}\label{lema_basic}
		Let $(H,A)$ be a \textit{compatible weak matched pair}. Then, for all $h\in H$,
		\begin{enumerate}
			\item[(i)] $h\cdot 1_A=(h_1\cdot 1_A)\varepsilon_s(h_2\cdot 1_A)$,
			\item[(ii)] $\varepsilon_s(h\cdot 1_A)=\varepsilon_s(h_1\cdot 1_A)\varepsilon_s(h_2\cdot 1_A)$,
			\item[(iii)] $h\cdot 1_{A1}\otimes 1_{A2}=(h\cdot 1_A)1_{A1}\otimes 1_{A2}$,
			\item[(iv)] $h\cdot 1_A=(h\cdot 1_{A1})\varepsilon_s(1_{A2})$.
		\end{enumerate}
	\end{lem}
	\begin{proof}
		We verify each item, as follows.
		
		$\underline{(i):}$ By items (i) and (ii) of Definition \ref{weak matched pair}, and the definition of $\varepsilon_s$,
		\begin{eqnarray*}
			h\cdot 1_A&=&h_1\cdot 1_{A1}\varepsilon(h_2\cdot 1_{A2})\\
			&=& (h_1\cdot 1_A) 1_{A1}\varepsilon((h_2\cdot 1_A) 1_{A2})\\
			&=& (h_1\cdot 1_A)\varepsilon_s(h_2\cdot 1_A).
		\end{eqnarray*}
		
		$\underline{(ii):}$ Just apply $\varepsilon_s$ on the both sides of the equality in item $(i)$ of Definition \ref{weak matched pair}, and use equality \eqref{4.20}.
		
		$\underline{(iii):}$ \begin{eqnarray*}
			h\cdot 1_{A1}\otimes 1_{A2}&=&h\cdot 1_{A1}\otimes 1_H\cdot 1_{A2}\\
			&=& (h\cdot 1_A) 1_{A1}\otimes (1_H\cdot 1_A) 1_{A2}\\
			&=& (h\cdot 1)1_{A1}\otimes 1_{A2}.
		\end{eqnarray*}
		
		$\underline{(iv):}$ Straightforward from item (ii) of Definition \ref{weak matched pair}.
	\end{proof}
	
	From now on, we always consider $(H,A)$ a compatible weak matched pair.
	
	\begin{pro}\label{copro_barrabaixo}
		The coproduct $\Delta$ is coassociative in $A \underline{\#} H$.   
	\end{pro}
	\begin{proof} Since $A \# H$ is coassociative, to prove the coassociativity of $A \underline{\#} H$ is enough to show that the coproduct $\Delta$ satisfies $\Delta(A \underline{\#}H) \subseteq A \underline{\#}H \otimes A \underline{\#}H$.
		Let $x \underline{\#} h\in A \underline{\#} H$,
		\begin{eqnarray*}
			\Delta (x \underline{\#} h)&=& \Delta (x (h_1 \cdot 1_A) \# h_2)\\
			&=& x_1 {(h_1 \cdot 1_A)}_{1} \# {h_2}^{0} \otimes x_2 {(h_1 \cdot 1_A)}_{2}{h_2}^{1} \# h_3\\
			&\stackrel{\ref{troca}}{=}& x_1 ({h_1}^{0} \cdot {1_A}_1) \# {h_3}^{0}{1_H}^{0} \otimes x_2 {h_1}^{1}(h_2 \cdot {1_A}_2){h_3}^{1}(h_4 \cdot {1_H}^{1}) \# h_5\\
			&\stackrel{\ref{weak matched pair}(ii)}{=}& x_1 ({h_1}^{0} \cdot {1_A}){1_A}_1 \# {h_3}^{0}{1_H}^{0} \otimes x_2 {h_1}^{1}(h_2 \cdot {1_A}){1_A}_2{h_3}^{1}(h_4 \cdot {1_H}^{1}) \# h_5\\
			&\stackrel{\ref{lema_basic}(i)}{=}& x_1 ({{h_1}^{0}}_1 \cdot {1_A}) \varepsilon_s({{h_1}^{0}}_2 \cdot 1_A){1_A}_1 \# {h_3}^{0}{1_H}^{0} \otimes x_2 {h_1}^{1}(h_2 \cdot {1_A}) \\
			& \ &{1_A}_2{h_3}^{1}(h_4 \cdot {1_H}^{1}) \# h_5\\
			&=& x_1 ({{h_1}^{0}} \cdot {1_A}) \varepsilon_s(({h_2}^{0}\cdot 1_A) {1_A}_1) \# {h_4}^{0}{1_H}^{0} \otimes x_2 {h_1}^{1}{h_2}^{1}(h_3 \cdot {1_A})\\
			& \ &{1_A}_2{h_4}^{1}(h_5 \cdot {1_H}^{1}) \# h_6\\
			&\stackrel{\ref{weak matched pair}(ii)}{=}& x_1 ({{h_1}^{0}} \cdot {1_A}) \varepsilon_s({h_2}^{0} \cdot {1_A}_1) \# {h_4}^{0}{1_H}^{0} \otimes x_2 {h_1}^{1}{h_2}^{1}(h_3 \cdot {1_A}_2) \\
			& \ & {h_4}^{1}(h_5 \cdot {1_H}^{1}) \# h_6\\
			&=& x_1 ({{h_1}^{0}} \cdot {1_A}) \varepsilon_s({({h_2} \cdot 1_A)}_1) \# {h_3}^{0} \otimes x_2 {h_1}^{1}{({h_2} \cdot 1_A)}_2{h_3}^{1} \# h_4\\
			&=& x_1 ({{h_1}^{0}} \cdot {1_A}) {1_A}_1 \# {h_3}^{0} \otimes x_2 {h_1}^{1}({h_2} \cdot 1_A){1_A}_2{h_3}^{1} \# h_4\\
			&=&x_1 {1_A}_1 ( {{h_1}^{0}} \cdot 1_A) \# {{h_2}^{0}} \otimes x_2 {1_A}_2{h_1}^{1}{h_2}^{1}({h_3} \cdot 1_A) \# h_4\\
			&=&x_1 ( {{h_1}^{0}}_{1} \cdot 1_A) \# {{h_1}^{0}}_{2} \otimes x_2 {h_1}^{1}({h_2} \cdot 1_A) \# h_3\\
			&=& x_1 \underline{\#} {h_1}^{0}\otimes x_2 {h_1}^{1}\underline{\#} {h_2}.
		\end{eqnarray*}
		Therefore, $\Delta (x \underline{\#} h)=x_1 \underline{\#} {h_1}^{0}\otimes x_2 {h_1}^{1}\underline{\#} {h_2}.$
	\end{proof}
	
	\medskip
	
	Note that $\varepsilon(x \underline{\#} h) = \varepsilon(x (h_1 \cdot 1_A) \# h_2) = \varepsilon(x (h_1 \cdot 1_A)) \varepsilon(h_2) =  \varepsilon(x (h \cdot 1_A))$, for all $x \underline{\#} h$ in $A \underline{\#} H$. However, it is easy to check that $\varepsilon$ is neither right nor left counit in $A \underline{\#} H$. Then, we consider the subspace $\overline{A \underline{\#} H}$ and, 
	\begin{eqnarray*}
		\overline{x \underline{\#} h}&=& (id \otimes \varepsilon)\Delta(x \underline{\#} h)\\
		&\stackrel{\ref{copro_barrabaixo}}{=}&x_1 \underline{\#} {h_1}^{0} \varepsilon( x_2 {h_1}^{1}\underline{\#} {h_2})\\
		&=&x_1 \underline{\#} {h_1}^{0} \varepsilon( x_2 {h_1}^{1}( h_2 \cdot 1_A))\\
		&=&x_1 ({{h_1}^{0}}_1 \cdot 1_A) \# {{h_1}^{0}}_2 \varepsilon( x_2 {h_1}^{1}( h_2 \cdot 1_A))\\
		&=&x_1 ({h_1}^{0} \cdot 1_A) \# {h_2}^{0} \varepsilon( x_2 {h_1}^{1}{h_2}^{1}( h_3 \cdot 1_A))\\
		&\stackrel{(\ref{4.15})}{=}& x \varepsilon_s({h_1}^{1}{h_2}^{1}( h_3 \cdot 1_A))({h_1}^{0} \cdot 1_A) \# {h_2}^{0},
	\end{eqnarray*}
	for all $\overline{x \underline{\#} h}\in \overline{A \underline{\#} H}$.
	
	\begin{pro}\label{copro_barrabarra}The coproduct $\Delta$ satisfies 
		$\Delta( \overline{x \underline{\#} h}) = \overline{x_1 \underline{\#} {h_1}^{0}} \otimes \overline{x_2 {h_1}^{1} \underline{\#} h_2}$, for all $ \overline{x \underline{\#} h}\in \overline{A \underline{\#} H}$.
	\end{pro}
	
	\begin{proof} Consider $ \overline{x \underline{\#} h}\in \overline{A \underline{\#} H}$. On the one hand,
		\begin{eqnarray*}
			\Delta (\overline{x \underline{\#} h})&=& \Delta (x_1 \underline{\#} {h_1}^{0}) \varepsilon( x_2 {h_1}^{1}( h_2 \cdot 1_A))\\
			&\stackrel{\ref{copro_barrabaixo}}{=}& x_1 \underline{\#} {{{h_1}^{0}}_1}^0 \otimes x_2 {{{h_1}^{0}}_1}^1 \underline{\#} {{h_1}^{0}}_2 \varepsilon ( x_3 {h_1}^1 ( h_2 \cdot 1_A) ) \\
			&=&  x_1 \underline{\#} {h_1}^{00} \otimes  x_2 {h_1}^{01} \underline{\#} {h_2}^{0} \varepsilon( x_3 {h_1}^1 {h_2}^1( h_3 \cdot 1_A))\\
			&=&  x_1 \underline{\#} {h_1}^0 \otimes x_2 {{h_1}^{1}}_1 \underline{\#} {h_2}^0 \varepsilon( x_3 {{h_1}^{1}}_{2}{h_2}^1 ( h_3 \cdot 1_A ) )\\
			&=&  x_1 \underline{\#} {h_1}^{0} \otimes  x_2 {h_1}^1 \underline{\#} {h_2}^{0} \varepsilon( x_3 {h_2}^{1}( h_3 \cdot 1_A)).
		\end{eqnarray*}
		
		On the other hand,
		\begin{eqnarray*}
			& & \overline{x_1 \underline{\#} {h_1}^{0}} \otimes \overline{x_2 {h_1}^{1} \underline{\#} h_2}\\
			&=&x_1 \underline{\#} {{{h_1}^{0}}_1}^{0} \varepsilon (x_2 {{{h_1}^{0}}_1}^{1} ({{{h_1}^{0}}_2} \cdot 1_A) ) \otimes x_3 {{h_1}^{1}}_1 \underline{\#} {h_2}^{0} \varepsilon (x_4 {{h_1}^{1}}_2 {h_2}^{1}(h_3 \cdot 1_A))\\
			&=&x_1 \underline{\#} {{{h_1}^{0}}_1}^{0} \varepsilon (x_2 {{{h_1}^{0}}_1}^{1} ({{{h_1}^{0}}_2} \cdot 1_A) ) \otimes x_3 {{h_1}^{1}}_1 \underline{\#} {h_2}^{0} \varepsilon (x_4 {{h_1}^{1}}_2 {h_2}^{1}(h_3 \cdot 1_A))\\
			&=&x_1 \underline{\#} {{{h_1}^{00}}} \varepsilon (x_2 {{{h_1}^{01}}} ({{{h_2}^{0}}} \cdot 1_A) ) \otimes x_3 {{h_1}^{1}} {{h_2}^{1}} \underline{\#} {h_3}^{0} \varepsilon (x_4 {h_3}^{1}(h_4 \cdot 1_A))\\
			&=&x_1 \underline{\#} {{{h_1}^{0}}} \varepsilon (x_2 {{{h_1}^{1}}_1}({{{h_2}^{0}}} \cdot 1_A) ) \otimes x_3 {{{h_1}^{1}}_2} {{h_2}^{1}} \underline{\#} {h_3}^{0} \varepsilon (x_4 {h_3}^{1}(h_4 \cdot 1_A))\\
			&=& x_1 \underline{\#} {{{h_1}^{0}}} \varepsilon (x_2 ({{{h_2}^{0}}} \cdot 1_A) ) \varepsilon (x_3 {{{h_1}^{1}}_1} )\otimes x_3 {{{h_1}^{1}}_2} {{h_2}^{1}} \underline{\#} {h_3}^{0} \varepsilon (x_4 {h_3}^{1}(h_4 \cdot 1_A))\\
			&=&\varepsilon (x_2 ({{{h_2}^{0}}} \cdot 1_A) )x_1 \underline{\#} {{{h_1}^{0}}}  \otimes x_3 {{{h_1}^{1}}} {{h_2}^{1}} \underline{\#} {h_3}^{0} \varepsilon (x_4 {h_3}^{1}(h_4 \cdot 1_A))\\
			&=&x_1 \varepsilon_s ({{{h_2}^{0}}} \cdot 1_A) \underline{\#} {h_1}^{0} \otimes x_2 {{{h_1}^{1}}} {{h_2}^{1}} \underline{\#} {h_3}^{0} \varepsilon (x_3 {h_3}^{1}(h_4 \cdot 1_A))\\
			&=& x_1 \varepsilon_s ({{h_1}^{0}}_2 \cdot 1_A) \underline{\#} {{h_1}^{0}}_1 \otimes x_2 {{{h_1}^{1}}} \underline{\#} {h_2}^{0} \varepsilon (x_3 {h_2}^{1}(h_3 \cdot 1_A))\\
			&=& x_1 \varepsilon_s ({{h_1}^{0}}_3 \cdot 1_A) ({{h_1}^{0}}_1 \cdot 1_A)\# {{h_1}^{0}}_2 \otimes x_2 {{{h_1}^{1}}} \underline{\#} {h_2}^{0} \varepsilon (x_3 {h_2}^{1}(h_3 \cdot 1_A))\\
			&=&x_1  ({{h_1}^{0}}_1 \cdot 1_A) \varepsilon_s ({{h_1}^{0}}_2 \cdot 1_A)\# {{h_1}^{0}}_3 \otimes x_2 {{{h_1}^{1}}} \underline{\#} {h_2}^{0}  \varepsilon (x_3 {h_2}^{1}(h_3 \cdot 1_A))\\
			&=&x_1  ({{h_1}^{0}}_1 \cdot 1_A)\# {{h_1}^{0}}_2 \otimes x_2 {{{h_1}^{1}}} \underline{\#} {h_2}^{0} \varepsilon (x_3 {h_2}^{1}(h_3 \cdot 1_A))\\
			&=&x_1 \underline{\#} {{h_1}^{0}} \otimes x_2 {{{h_1}^{1}}} \underline{\#} {h_2}^{0} \varepsilon (x_3 {h_2}^{1}(h_3 \cdot 1_A)).
		\end{eqnarray*}
		Therefore, $\Delta( \overline{x \underline{\#} h}) = \overline{x_1 \underline{\#} {h_1}^{0}} \otimes \overline{x_2 {h_1}^{1} \underline{\#} h_2} \in \overline{A \underline{\#} H} \otimes \overline{A \underline{\#} H}$.
	\end{proof}
	
	Now, in order to prove that $\overline{A \underline{\#} H}$ has a counit, we observe the following.
	
	\begin{rem} 
		Considering the equalities in the proof of Proposition \ref{copro_barrabaixo}, we conclude, by applying  $id\otimes id\otimes id\otimes\varepsilon$, that
		\begin{equation}\label{**}
			x_1 ( {{h_1}^{0}} \cdot 1_A) \# {{h_2}^{0}} \otimes x_2 {h_1}^{1}{h_2}^{1}({h_3} \cdot 1_A) = x_1 {(h_1 \cdot 1_A)}_{1} \# {h_2}^{0} \otimes x_2 {(h_1 \cdot 1_A)}_{2}{h_2}^{1}.
		\end{equation}
		Besides that, applying $(\varepsilon \otimes \varepsilon \otimes id)$ in  the above equality, and considering $x=1_A$ and $H$ cocommutative, we obtain:
		\begin{equation}\label{triangulo}
			\varepsilon ( {{h_1}^{0}}) ( {{h_2}} \cdot 1_A) {h_1}^1 = \varepsilon_t ({h_1}^0  \cdot 1_A) {h_1}^{1} (h_2 \cdot 1_A).
		\end{equation}
	\end{rem}

	\begin{pro}
		The subspace $\overline{A \underline{\#} H}$ is a coalgebra.
	\end{pro}
	
	\begin{proof}
		First, let's check that $\varepsilon$ is a counit em  $\overline{A \underline{\#} H}$.
		Indeed, recall that $\varepsilon(\overline{x \underline{\#} h})= \varepsilon (x ({h_1}^0 \cdot 1_A){h_1}^1(h_2 \cdot 1_A))$, and so
		\begin{eqnarray*}
			& & (id \otimes \varepsilon)\Delta (\overline{x \underline{\#} h})\\
			&=& \overline{x_1 \underline{\#} {h_1}^{0}}\varepsilon (\overline{x_2 {h_1}^{1} \underline{\#} h_2})\\
			&=& \overline{x_1 \underline{\#} {h_1}^{0}} \varepsilon (x_2 {h_1}^{1} ({h_2}^0\cdot 1_A) {h_2}^1 ({h_3} \cdot 1_A) )\\
			&=& x_1 \underline{\#} {{{h_1}^{0}}_1}^0 \varepsilon (x_2 {{{h_1}^{0}}_1}^1 ({{h_1}^{0}}_2\cdot 1_A)) \varepsilon(x_3 {h_1}^1 {h_2}^1 ({h_2}^0 \cdot 1_A) ({h_3} \cdot 1_A) )\\
			&=& x_1 \underline{\#} {{{h_1}^{00}}} \varepsilon (x_2 {{{h_1}^{01}}} ({{h_2}^{0}}\cdot 1_A)) \varepsilon(x_3 {h_1}^1 {h_2}^1 {h_3}^1 ({h_3}^0 \cdot 1_A) ({h_4} \cdot 1_A) )\\
			&=& x_1 \underline{\#} {h_1}^{0} \varepsilon (x_2 {{{h_1}^{1}}_1} ({{h_2}^{0}}\cdot 1_A)) \varepsilon(x_3 {{{h_1}^{1}}_2} {h_2}^1 {h_3}^1 ({h_3}^0 \cdot 1_A) ({h_4} \cdot 1_A) )\\
			&=& x_1 \underline{\#} {h_1}^{0}  \varepsilon(x_2 {h_1}^{1} {h_2}^1 {h_3}^1 ({h_2}^0 \cdot 1_A) ({h_3}^0 \cdot 1_A) ({h_4} \cdot 1_A) )\\
			&=& x_1 \underline{\#} {h_1}^{0}  \varepsilon(x_2 {h_1}^{1} {h_2}^1  ({{h_2}^0} \cdot 1_A) ({h_3} \cdot 1_A) )\\
			&=& x_1 \underline{\#} {{h_1}^{0}}_1  \varepsilon(x_2 {h_1}^{1}  ({{h_1}^{0}}_2 \cdot 1_A) ({h_2} \cdot 1_A) )\\
			&\stackrel{\eqref{4.15}}{=}&  x_1 \varepsilon_s ({{h_1}^{0}}_2 \cdot 1_A)  \underline{\#} {{h_1}^{0}}_1  \varepsilon(x_2 {h_1}^{1}  ({h_2} \cdot 1_A)) \\
			&=& x_1 \varepsilon_s ({{h_1}^{0}}_3 \cdot 1_A)  ({{h_1}^{0}}_1 \cdot 1_A)  \# {{h_1}^{0}}_2  \varepsilon(x_2 {h_1}^{1}  ({h_2} \cdot 1_A)) \\
			&\stackrel{\ref{lema_basic}(i)}{=}& x_1   ({{h_1}^{0}}_1 \cdot 1_A)  \# {{h_1}^{0}}_2  \varepsilon(x_2 {h_1}^{1}  ({h_2} \cdot 1_A)) \\
			&=& x_1  \underline{\#} {{h_1}^{0}}  \varepsilon(x_2 {h_1}^{1}  ({h_2} \cdot 1_A)) \\
			&=&\overline{x \underline{\#} h},
		\end{eqnarray*}
		and also
		\begin{eqnarray*}
			(\varepsilon \otimes id)\Delta (\overline{x \underline{\#} h})&=&\varepsilon( \overline{x_1 \underline{\#} {h_1}^{0}}) \overline{x_2 {h_1}^{1} \underline{\#} h_2}\\
			&=&\varepsilon(x_1 ({{{h_1}^{0}}_1}^0 \cdot 1_A) {{{h_1}^{0}}_1}^1({{h_1}^{0}}_2 \cdot 1_A) )\overline{x_2 {h_1}^{1} \underline{\#} h_2}\\
			&=&\varepsilon(x_1 ({{{h_1}^{00}}} \cdot 1_A) {{{h_1}^{01}}}({{h_2}^{0}} \cdot 1_A) )\overline{x_2 {h_1}^{1}{h_2}^{1} \underline{\#} h_3}\\
			&=&\varepsilon(x_1 ({{{h_1}^{0}}} \cdot 1_A) {{{h_1}^{1}}_1}({{h_2}^{0}} \cdot 1_A) )\overline{x_2 {{{h_1}^{1}}_2}{h_2}^{1} \underline{\#} h_3}\\
			&=&\overline{x {{h_1}^{1}}{h_2}^{1} \varepsilon_t( ({{{h_1}^{0}}} \cdot 1_A)({{h_2}^{0}} \cdot 1_A) ) \underline{\#} h_3}\\
			&=&\overline{x {{h_1}^{1}} \varepsilon_t( ({{{h_1}^{0}}_1} \cdot 1_A)({{{h_1}^{0}}_2} \cdot 1_A) ) \underline{\#} h_2}\\
			&=&\overline{x {{h_1}^{1}} \varepsilon_t( ({{{h_1}^{0}}} \cdot 1_A)\underline{\#} h_2}\\
			&=& (id \otimes \varepsilon) \Delta (x {{h_1}^{1}} \varepsilon_t( {{{h_1}^{0}}} \cdot 1_A)\underline{\#} h_2)\\
			&=& (id \otimes \varepsilon) \Delta (x {{h_1}^{1}} \varepsilon_t({{{h_1}^{0}}} \cdot 1_A) (h_2 \cdot 1_A )\# h_3)\\
			&\stackrel{(\ref{triangulo})}{=}& (id \otimes \varepsilon) \Delta (x {{h_1}^{1}} \varepsilon({h_1}^{0}) (h_2 \cdot 1_A )\# h_3)\\
			&=&\overline{x \varepsilon({h_1}^0){h_1}^1 \underline{\#} h_2}\\
			&=&\overline{x \varepsilon_s({h_1}^1) \underline{\#} h_2 \varepsilon({h_1}^0)}\\
			&=&x_1 \varepsilon_s({h_1}^1)_1 \underline{\#} {h_2}^0 \varepsilon(x_2 \varepsilon_s({h_1}^1)_2  {h_2}^1 (h_3 \cdot 1_A))\varepsilon ({h_1}^0)\\
			&=&x_1 {1_A}_1 \underline{\#} {h_2}^0 \varepsilon(x_2 {1_A}_2 \varepsilon_s({h_1}^1)  {h_2}^1 (h_3 \cdot 1_A))\varepsilon ({h_1}^0)\\
			&=&x \varepsilon_s( \varepsilon_s({h_1}^1)  {h_2}^1 (h_3 \cdot 1_A)) \underline{\#} {h_2}^0 \varepsilon ({h_1}^0)\\
			&=&x \varepsilon_s({h_1}^1 {h_2}^1 (h_3 \cdot 1_A)) \underline{\#} {h_2}^0 \varepsilon ({h_1}^0)\\
			&=&x \varepsilon_s({h_1}^1 (h_2\cdot 1_A)) \underline{\#} {{h_1}^0}_2 \varepsilon ({{h_1}^0}_1)\\
			&=& \overline{x \underline{\#} h}.
		\end{eqnarray*}
		Therefore, by Proposition \ref{copro_barrabarra}$, \overline{A\underline{\#}H}$ is a coalgebra.
	\end{proof}
	
	To verify that $\overline{A\underline{\#}H}$ is also a subalgebra of $A\underline{\#}H$, we observe that the order in which we perform the projections does not matter. 
	
	\begin{rem} \label{obs_ordem} Denote by $\underline{x\overline{\#}h}=((id \otimes\varepsilon)\Delta(x\#h))(1_A\# 1_H)$. Then,
		$\overline{x\underline{\#}h}=\underline{x\overline{\#}h}$.
		Indeed,
		\begin{eqnarray*}
			\overline{x\underline{\#}h}&=& x_1 ({{h_1}^0}\cdot 1_A) \#  {{h_2}^0}\varepsilon(x_2{h_1}^1 {h_2}^1(h_3\cdot 1_A))\\
			&\stackrel{\eqref{**}}{=}& x_1(h_1\cdot 1_A)_1\# {h_2}^0\varepsilon(x_2(h_1\cdot 1_A)_2{h_2}^1)\\
			&=& x_1(h_1\cdot 1_A)_1\varepsilon(x_2(h_1\cdot 1_A)_2)\# {h_2}^0\varepsilon(x_3{h_2}^1)\\
			&=& x_1(h_1\cdot 1_A)\# {h_2}^0\varepsilon(x_2{h_2}^1)\\
			&=& x_1({h^0}_1\cdot 1_A)\# {{h^0}_2}^0\varepsilon(x_2{{h^0}_2}^1)\varepsilon(h^1)\\
			&=& x_1({h_1}^0\cdot 1_A)\# {h_2}^{00}\varepsilon(x_2{h_2}^{01})\varepsilon({h_1}^1{h_2}^1)\\
			&=& x_1({h_1}^0\cdot 1_A)\# {h_2}^{0}\varepsilon(x_2{{h_2}^{1}}_1)\varepsilon({h_1}^1{{h_2}^1}_2)\\
			&=& x_1({h_1}^0\cdot 1_A)\# {h_2}^{0}\varepsilon(x_2{h_1}^{1}{h_2}^1)\\
			&=& x_1({h^0}_1\cdot 1_A)\# {h^0}_2\varepsilon(x_2{h}^{1})\\
			&=& (x_1\varepsilon(x_2{h}^{1})\# h^0)(1_A\# 1_H)\\
			&=& ((id \otimes\varepsilon)\Delta(x\#h))(1_A\# 1_H).
		\end{eqnarray*}
		
		Therefore, we write the element $\overline{x\underline{\#}h}=\underline{x\overline{\#}h}$ simply by $x\due h$ and $\overline{A\underline{\#}H}$ for $A\due H$.
		Furthermore, it is sometimes convenient to handle writing one projection explicitly, \emph{i.e.} $x\due h = x_1 \varepsilon(x_2 h^1) \underline{\#} h^0 = x(h_1 \cdot 1_A) \overline{\#} h_2$.
	\end{rem}
	
	The next remark will be very useful for the proof of some results in the following, it has some identities for the product of elements in the argument of the map $\varepsilon$.
	\begin{rem}
		Since $A$ is commutative, by equality \eqref{4.6}, for $x,y$ and $z \in A$, we get 
		\begin{equation*}
			\varepsilon(xyz)=\varepsilon(\varepsilon_s(x)yz)=\varepsilon(\varepsilon_s(xy)z)=\varepsilon(\varepsilon_s(xyz))=\varepsilon(x\varepsilon_s(yz))=\varepsilon(xy\varepsilon_s(z)).
		\end{equation*}
	\end{rem}
	
	\begin{pro}
		$A\due H$ is an algebra, with unity $1_A\due 1_H$.
	\end{pro}
	\begin{proof} Let $(x\due h), (y\due g)\in A\due H$. First, we note that $(x\due h)(y\due g)=x(h_1\cdot y)\due h_2g$. Indeed,
		\begin{eqnarray*}
			& \ & x(h_1\cdot y)\due h_2g \\
			&=& x_1(h_1\cdot y)_1\underline{\#}(h_2g_1)^0\varepsilon(x_2(h_1\cdot y)_2(h_2g_1)^1(h_3g_2\cdot 1_A))\\
			&\stackrel{\ref{troca}}{=}& x_1({h_1}^0\cdot y_1)\underline{\#} {h_2}^0{g_1}^0\varepsilon(x_2{h_1}^1{h_2}^1(h_3\cdot y_2{g_1}^1)\varepsilon_s(h_4\cdot(g_2\cdot 1_A)))\\
			&\stackrel{\ref{novo weak matched pair}(iv)}{=}& x_1({h_1}^0\cdot y_1)\underline{\#} {h_2}^0{g_1}^0\varepsilon(x_2{h_1}^1{h_2}^1(h_3\cdot y_2{g_1}^1)\varepsilon_s(h_4\cdot 1_A)\varepsilon_s(g_2\cdot 1_A))\\
			&=& x_1({h_1}^0\cdot y_1)\underline{\#} {h_2}^0{g_1}^0\varepsilon(x_2{h_1}^1{h_2}^1(h_3\cdot y_2{g_1}^1)(g_2\cdot 1_A))\\
			&=& x_1({{h_1}^0}_1\cdot y_1)\underline{\#} {{h_2}^0}_2{g_1}^0\varepsilon(x_2{h_1}^1(h_2\cdot y_2{g_1}^1)(g_2\cdot 1_A))\\
			&\stackrel{\eqref{4.15}}{=}&x\varepsilon_s({h_1}^1(h_2\cdot y_2{g_1}^1)(g_2\cdot 1_A))({{h_1}^0}_1\cdot y_1)\underline{\#} {{h_2}^0}_2{g_1}^0\\
			&\stackrel{\ref{novo weak matched pair}(iv)}{=}& x\varepsilon_s({h_1}^1)\varepsilon_s(h_2\cdot 1_A)\varepsilon_s(y_2{g_1}^1(g_2\cdot 1_A))({{h_1}^0}_1\cdot y_1)\underline{\#} {{h_2}^0}_2{g_1}^0\\
			&\stackrel{\eqref{4.15}}{=}&\varepsilon_s({h_1}^1(h_2\cdot 1_A))x_1 ({{h_1}^0}_1\cdot y_1\varepsilon(x_2y_2{g_1}^1(g_2\cdot 1_A)))\underline{\#} {{h_2}^0}_2{g_1}^0\\
			&\stackrel{\eqref{4.15}}{=}&\varepsilon_s({h_1}^1(h_2\cdot 1_A))x_1 ({{h_1}^0}_1\cdot y\varepsilon_s(x_2)\varepsilon_s({g_1}^1(g_2\cdot 1_A)))\underline{\#} {{h_2}^0}_2{g_1}^0\\
			&=& \varepsilon_s({h_1}^1(h_2\cdot 1_A))x_1\varepsilon(x_21_{A2})({{h_1}^0}_1\cdot 1_{A1}y\varepsilon_s({g_1}^1(g_2\cdot 1_A)))\underline{\#} {{h_2}^0}_2{g_1}^0\\
			&\stackrel{\eqref{4.15}}{=}& \varepsilon_s({h_1}^1(h_2\cdot 1_A))x\varepsilon_s(1_{A2})({{h_1}^0}_1\cdot 1_{A1})({{h_1}^0}_2\cdot y\varepsilon_s({g_1}^1(g_2\cdot 1_A)))\underline{\#} {{h_2}^0}_3{g_1}^0\\
			&\stackrel{\ref{lema_basic}(iii)}{=}& \varepsilon_s({h_1}^1(h_2\cdot 1_A))x1_{A1}\varepsilon_s(1_{A2})({{h_1}^0}_1\cdot 1_A)({{h_1}^0}_2\cdot y\varepsilon_s({g_1}^1(g_2\cdot 1_A)))\underline{\#} {{h_2}^0}_3{g_1}^0\\
			&=& \varepsilon_s({h_1}^1(h_2\cdot 1_A))x ({{h_1}^0}_1\cdot y\varepsilon_s({g_1}^1(g_2\cdot 1_A)))\underline{\#} {{h_2}^0}_2{g_1}^0\\
			&\stackrel{\eqref{4.15}}{=}& x_1 ({{h_1}^0}_1\cdot y_1)\underline{\#} {{h_2}^0}_2{g_1}^0 \varepsilon(x_2{h_1}^1(h_2\cdot 1_A))\varepsilon(y_2{g_1}^1(g_2\cdot 1_A))\\
			&=& (x_1\underline{\#} {h_1}^0)(y_1\underline{\#} {g_1}^0) \varepsilon(x_2{h_1}^1(h_2\cdot 1_A))\varepsilon(y_2{g_1}^1(g_2\cdot 1_A))\\
			&=& (x\due h)(y\due g).
		\end{eqnarray*}
		
		Now, we verify that $1_A \due 1_H$ is the unity. Indeed,
		\begin{eqnarray*}
			(x\due h)(1_A\due 1_H)=x(h_1\cdot 1_A)\due h_2= \overline{x(h_1\cdot 1_A)\underline{\#} h_2}=\overline{(x\underline{\#} h)(1_A\underline{\#} 1_H)}=\overline{x\underline{\#} h}.
		\end{eqnarray*}
		Analogously, $(1_A\due 1_H)(x\due h)=x\due h$. 
		Thus, $A \due H$ is a subalgebra of $A\underline{\#}H$.
	\end{proof}
	
	Then, we conclude that $A\due H$ is a weak bialgebra, as follows.
	
	\begin{thm}\label{principal}
		Let $(H,A)$ be a compatible weak matched pair. Then, $A\overline{\underline{\#}} H$ is a weak bialgebra with the following structure:
		\begin{enumerate}
			\item[(i)] Multiplication: $(x\overline{\underline{\#}} h)(y \overline{\underline{\#}} g)= x(h_1\cdot y)\overline{\underline{\#}} h_2 g$,
			\item[(ii)] Unity: $1_A \overline{\underline{\#}} 1_H$,
			\item[(iii)] Comultiplication: $\Delta(x\overline{\underline{\#}}h)=x_1 \overline{\underline{\#}} {h_1}^0 \otimes x_2 {h_1}^1 \overline{\underline{\#}} h_2$,
			\item[(iv)] Counit: $\varepsilon(x\overline{\underline{\#}} h)= \varepsilon(x({h_1}^0\cdot 1_A){h_1}^1(h_2\cdot 1_A))$,
		\end{enumerate}
		for all $x, y\in A$ and $h,g\in H$.
	\end{thm}
	
	\begin{proof} We already know that $A\due H$ is an algebra and a coalgebra and, by Proposition \ref{delta_mult}, the map $\Delta$ is multiplicative. Then, it is enough to verify the properties of $\varepsilon$ and $\Delta(1_A\due 1_H)$, \emph{i.e.}, items (ii) and (iii) of the Definition \ref{def_weakbialgebra}:
		
		\medskip
		
		Let $x \due h, y \due g, z \due t\in A\overline{\underline{\#}} H$. Then,
		\begin{eqnarray*}
			& \ &\varepsilon\left((x \due h)(y \due g)(z \due t)\right)\\
			&=&\varepsilon\left(\underline{\overline{(x \# h)(y \# g)(z \# t)}}\right)\\
			&=&\varepsilon\left(\underline{\overline{(x(h_1 \cdot y) \# h_2g)}} \underline{\overline{(z \# t)}} \right)\\
			&=&\varepsilon\left(x(h_1 \cdot y) (h_2g \cdot z) \due h_3g_2t\right)\\
			&=&\varepsilon(x(h_1 \cdot y) (h_2g_1 \cdot z) ((h_3g_2t_1)^0 \cdot 1_A) (h_3g_2t_1)^1 (h_4g_3t_2 \cdot 1_A))\\
			&=&\varepsilon(x ((h_2g_2t_1)^0 \cdot 1_A) (h_1 \cdot y)_2 (h_2g_2t_1)^1 (h_3g_1 \cdot z)   (h_4g_3t_2 \cdot 1_A))\varepsilon((h_1 \cdot y)_1)\\
			&\stackrel{\ref{troca}}{=}&\varepsilon(x ({h_2}^0 (g_2t_1)^0 \cdot 1_A) {h_1}^1{h_2}^1 (h_3 \cdot y_2) (h_4 \cdot (g_2t_1)^1) (h_5g_1 \cdot z)   (h_6g_3t_2 \cdot 1_A)) \\
			& \ & \varepsilon({h_1}^0 \cdot y_1)\\
			&=&\varepsilon(x ({h_2}^0 (g_2t_1)^0 \cdot 1_A) {h_1}^1{h_2}^1 (h_3 \cdot y_2) (h_4 \cdot (g_2t_1)^1 (g_1 \cdot z)) (h_5g_3t_2 \cdot 1_A))\\
			& \ & \varepsilon({h_1}^0 \cdot y_1)\\
			&=&\varepsilon(x ({h_2}^0 (g_2t_1)^0 \cdot 1_A) {h_1}^1{h_2}^1 (h_3 \cdot y_2) (h_4 \cdot (g_2t_1)^1 (g_1 \cdot z)_2) (h_5g_3t_2 \cdot 1_A)) \\
			& \ &\varepsilon({h_1}^0 \cdot y_1) \varepsilon((g_1 \cdot z)_1)\\
			&\stackrel{\ref{troca}}{=}&\varepsilon(x ({h_2}^0 {g_2}^0 {t_1}^0 \cdot 1_A) {h_1}^1{h_2}^1 (h_3 \cdot y_2) (h_4 \cdot {g_1}^1 {g_2}^1 (g_3 \cdot z_2{t_1}^1)) (h_5g_4t_2 \cdot 1_A))\\
			& \ & \varepsilon({h_1}^0 \cdot y_1) \varepsilon({g_1}^0 \cdot z_1)\\
			&\stackrel{(\ref{4.6})}{=}&\varepsilon(x \varepsilon_s({h_2}^0 {g_2}^0 {t_1}^0 \cdot 1_A) \varepsilon_s({h_1}^1{h_2}^1) \varepsilon_s(h_3 \cdot y_2) \varepsilon_s(h_4 \cdot ({g_1}^1 {g_2}^1 (g_3 \cdot z_2{t_1}^1))) \\
			& \ & \varepsilon_s(h_5g_4t_2 \cdot 1_A)) \varepsilon({h_1}^0 \cdot y_1) \varepsilon({g_1}^0 \cdot z_1)\\
			&\stackrel{\ref{novo weak matched pair}(iv)}{=}&\varepsilon(x \varepsilon_s({h_2}^0  \cdot 1_A) \varepsilon_s({g_2}^0  \cdot 1_A) \varepsilon_s({t_1}^0  \cdot 1_A) \varepsilon_s({h_1}^1{h_2}^1) \varepsilon_s(h_3 \cdot 1_A) \varepsilon_s(y_2)  \\
			&\ & \varepsilon_s(h_4 \cdot 1_A)  \varepsilon_s({g_1}^1 {g_2}^1) \varepsilon_s(g_3 \cdot 1_A) \varepsilon_s(z_2) \varepsilon_s({t_1}^1) \varepsilon_s(h_5 \cdot 1_A)\varepsilon_s(g_4 \cdot 1_A)  \\
			& \ & \varepsilon_s(t_2 \cdot 1_A)) \varepsilon({h_1}^0 \cdot y_1) \varepsilon({g_1}^0 \cdot z_1)\\
			&\stackrel{\ref{lema_basic}(ii)}{=}&\varepsilon(x \varepsilon_s({h_2}^0  \cdot 1_A) \varepsilon_s({g_2}^0  \cdot 1_A) \varepsilon_s({t_1}^0  \cdot 1_A) \varepsilon_s({h_1}^1{h_2}^1) \varepsilon_s({g_1}^1 {g_2}^1)   \varepsilon_s(y_2)   \\
			&\ &\varepsilon_s(h_3 \cdot 1_A)  \varepsilon_s(g_3 \cdot 1_A) \varepsilon_s(z_2) \varepsilon_s({t_1}^1)  \varepsilon_s(t_2 \cdot 1_A)) \varepsilon({h_1}^0 \cdot y_1) \varepsilon({g_1}^0 \cdot z_1)\\
			&\stackrel{\ref{novo weak matched pair}(iv)}{=}&\varepsilon(x \varepsilon_s({h_2}^0  \cdot y_2) \varepsilon_s({g_2}^0  \cdot z_2) \varepsilon_s({t_1}^0  \cdot 1_A) \varepsilon_s({h_1}^1{h_2}^1) \varepsilon_s({g_1}^1 {g_2}^1) \varepsilon_s({t_1}^1)     \\ 
			& \ & \varepsilon_s(h_3 \cdot 1_A) \varepsilon_s(g_3 \cdot 1_A)  \varepsilon_s(t_2 \cdot 1_A))
			\varepsilon({h_1}^0 \cdot y_1) \varepsilon({g_1}^0 \cdot z_1).
		\end{eqnarray*}
		
		On the other side,
		\begin{eqnarray*}
			& \ &\varepsilon \left((x \due h)(y \due g)_1 \right) \varepsilon \left((y \due g)_2(z \due t) \right)\\
			&=&\varepsilon \left( (x \due h)(y_1 \due {g_1}^{0}) \right) \varepsilon \left( (y_2{g_1}^{1} \due g_2)(z \due t) \right)\\
			&=&\varepsilon \left( x(h_1 \cdot y_1) \due h_2{g_1}^{0} \right) \varepsilon \left( y_2{g_1}^{1}(g_2 \cdot z) \due g_3t \right) \\
			&=&\varepsilon \left( x(h_1 \cdot y_1) ((h_2{{g_1}^{0}}_{1})^{0} \cdot 1_A) (h_2{{g_1}^{0}}_{1})^{1} (h_3{{g_1}^{0}}_{2} \cdot 1_A) \right)\\
			& \ & \varepsilon \left( y_2{g_1}^{1}(g_2 \cdot z) ((g_3t_1)^{0}\cdot 1_A) (g_3t_1)^{1}(g_4t_2 \cdot 1_A) \right)\\
			&=&\varepsilon((h_1 \cdot y_1)_1)\varepsilon(x (h_1 \cdot y_1)_2((h_2{{g_1}^{0}}_{1})^{0} \cdot 1_A) (h_2{{g_1}^{0}}_{1})^{1} (h_3{{g_1}^{0}}_{2} \cdot 1_A))\varepsilon((g_2\cdot z)_1)\\
			& \ &\varepsilon(y_2{g_1}^{1}(g_2 \cdot z)_2 ((g_3t_1)^{0}\cdot 1_A) (g_3t_1)^{1}(g_4t_2 \cdot 1_A))\\
			&\stackrel{\ref{troca}}{=}&\varepsilon({h_1}^{0} \cdot y_1)\varepsilon(x {h_1}^{1} {h_2}^{1} (h_3 \cdot y_2{{{g_1}^{0}}_{1}}^{1}) ({h_2}^{0} {{{g_1}^{0}}_{1}}^{0}\cdot 1_A) (h_4{{g_1}^{0}}_{2} \cdot 1_A))\varepsilon({g_2}^{0}\cdot z_1) \\
			& \ & \varepsilon(y_3{g_1}^{1}{g_2}^{1}{g_3}^{1}(g_4 \cdot z_2{t_1}^{1}) ({g_3}^{0}{t_1}^{0}\cdot 1_A)(g_5t_2 \cdot 1_A))\\
			&\stackrel{(\ref{4.6})}{=}&\varepsilon({h_1}^{0} \cdot y_1)\varepsilon(x \varepsilon_s({h_1}^{1} {h_2}^{1}) \varepsilon_s(h_3 \cdot y_2{{{g_1}^{0}}_{1}}^{1}) \varepsilon_s({h_2}^{0} \cdot 1_A) \varepsilon_s ({{{g_1}^{0}}_{1}}^{0}\cdot 1_A)\varepsilon_s(h_4 \cdot 1_A) \\
			& \ &\varepsilon_s({{g_1}^{0}}_{2}\cdot 1_A)) \varepsilon({g_2}^{0}\cdot z_1)\varepsilon(y_3{g_1}^{1}\varepsilon_s({g_2}^{1}{g_3}^{1})\varepsilon_s(g_4 \cdot z_2{t_1}^{1}) 
			\varepsilon_s({g_3}^{0}\cdot 1_A)\varepsilon_s({t_1}^{0}\cdot 1_A)\\
			& \ & \varepsilon_s(g_5 \cdot 1_A)\varepsilon_s(t_2 \cdot 1_A))\\
			&\stackrel{\ref{troca}}{=}&\varepsilon({h_1}^{0} \cdot y_1)\varepsilon(x \varepsilon_s({h_1}^{1} {h_2}^{1}) \varepsilon_s(h_3 \cdot 1_A) \varepsilon_s(y_2{{{g_1}^{0}}_{1}}^{1}) \varepsilon_s({h_2}^{0} \cdot 1_A) \varepsilon_s ({{{g_1}^{0}}_{1}}^{0}\cdot 1_A)\\
			& \ & \varepsilon_s(h_4 \cdot 1_A) \varepsilon_s({{g_1}^{0}}_{2}\cdot 1_A)) \varepsilon({g_2}^{0}\cdot z_1)\varepsilon(y_3{g_1}^{1}\varepsilon_s({g_2}^{1}{g_3}^{1})\varepsilon_s(g_4 \cdot 1_A) \varepsilon_s(z_2) \\
			& \ & \varepsilon_s({t_1}^{1}) 
			\varepsilon_s({g_3}^{0}\cdot 1_A)\varepsilon_s({t_1}^{0}\cdot 1_A)\varepsilon_s(g_5 \cdot 1_A)\varepsilon_s(t_2 \cdot 1_A))\\
			&=&\varepsilon({h_1}^{0} \cdot y_1)\varepsilon(x \varepsilon_s({h_1}^{1} {h_2}^{1}) \varepsilon_s(h_3 \cdot 1_A) \varepsilon_s(y_2{{g_1}^{01}}) \varepsilon_s({h_2}^{0} \cdot 1_A) \varepsilon_s ({{g_1}^{00}}\cdot 1_A) \\
			& \ &\varepsilon_s({{g_2}^{0}}\cdot 1_A)) \varepsilon({g_3}^{0}\cdot z_1)\varepsilon(y_3{g_1}^{1}{g_2}^{1}\varepsilon_s({g_3}^{1}{g_4}^{1})\varepsilon_s(g_5 \cdot 1_A) \varepsilon_s(z_2) \varepsilon_s({t_1}^{1}) 
			\\
			& \ & \varepsilon_s({g_4}^{0}\cdot 1_A) \varepsilon_s({t_1}^{0}\cdot 1_A)\varepsilon_s(t_2 \cdot 1_A))\\
			&=&\varepsilon({h_1}^{0} \cdot y_1)\varepsilon(x \varepsilon_s({h_1}^{1} {h_2}^{1}) \varepsilon_s(h_3 \cdot 1_A)  \varepsilon_s(y_2{{g_1}^{1}}_{1}) \varepsilon_s({h_2}^{0} \cdot 1_A) \varepsilon_s ({{g_1}^{0}}\cdot 1_A) \\
			& \ &\varepsilon_s({{g_2}^{0}}\cdot 1_A)) \varepsilon({g_3}^{0}\cdot z_1)\varepsilon(y_3{{g_1}^{1}}_2{g_2}^{1}\varepsilon_s({g_3}^{1}{g_4}^{1})\varepsilon_s(g_5 \cdot 1_A) \varepsilon_s(z_2) \varepsilon_s({t_1}^{1}) 
			\\
			& \ & \varepsilon_s({g_4}^{0}\cdot 1_A) \varepsilon_s({t_1}^{0}\cdot 1_A)\varepsilon_s(t_2 \cdot 1_A))\\
			&\stackrel{(\ref{4.6})}{=}&\varepsilon({h_1}^{0} \cdot y_1) \varepsilon({g_3}^{0}\cdot z_1) \varepsilon(x \varepsilon_s({h_1}^{1} {h_2}^{1}) \varepsilon_s(h_3 \cdot 1_A)  y_2{{g_1}^{1}} {{g_2}^{1}}\varepsilon_s({h_2}^{0} \cdot 1_A) \\ 
			& \ & \varepsilon_s ({{g_1}^{0}}\cdot 1_A) \varepsilon_s({{g_2}^{0}}\cdot 1_A)\varepsilon_s({g_3}^{1}{g_4}^{1}) \varepsilon_s(g_5 \cdot 1_A)  \varepsilon_s(z_2) \varepsilon_s({t_1}^{1}) 
			\varepsilon_s({g_4}^{0}\cdot 1_A)  \\
			& \ & \varepsilon_s({t_1}^{0}\cdot 1_A) \varepsilon_s(t_2 \cdot 1_A))\\
			&{=}&\varepsilon({h_1}^{0} \cdot y_1) \varepsilon({g_2}^{0}\cdot z_1) \varepsilon(x \varepsilon_s({h_1}^{1} {h_2}^{1}) \varepsilon_s(h_3 \cdot 1_A)  \varepsilon_s(y_2)\varepsilon_s({{g_1}^{1}})\varepsilon_s({h_2}^{0} \cdot 1_A) \\
			& \ & \varepsilon_s ({{g_1}^{0}}_1\cdot 1_A)  \varepsilon_s({{g_1}^{0}}_2\cdot 1_A) \varepsilon_s({g_2}^{1}{g_3}^{1}) \varepsilon_s(g_4 \cdot 1_A)  \varepsilon_s(z_2) \varepsilon_s({t_1}^{1}) 
			\varepsilon_s({g_3}^{0}\cdot 1_A) \\
			& \ & \varepsilon_s({t_1}^{0}\cdot 1_A)\varepsilon_s(t_2 \cdot 1_A))\\
			&{=}&\varepsilon({h_1}^{0} \cdot y_1) \varepsilon({g_1}^{0}\cdot z_1) \varepsilon(x \varepsilon_s({h_1}^{1} {h_2}^{1}) \varepsilon_s(h_3 \cdot 1_A)  \varepsilon_s({{g_1}^{1}}{{g_2}^{1}}{{g_3}^{1}})\varepsilon_s({h_2}^{0} \cdot y_2) \\
			& \ & \varepsilon_s ({{g_2}^{0}}\cdot 1_A) \varepsilon_s({{g_3}^{0}}\cdot 1_A) \varepsilon_s(g_4 \cdot 1_A)  \varepsilon_s(z_2) \varepsilon_s({t_1}^{1}) 
			\varepsilon_s({g_3}^{0}\cdot 1_A) \varepsilon_s({t_1}^{0}\cdot 1_A)\\
			& \ &\varepsilon_s(t_2 \cdot 1_A))\\
			&=&\varepsilon({h_1}^{0} \cdot y_1) \varepsilon({g_1}^{0}\cdot z_1) \varepsilon(x \varepsilon_s({h_1}^{1} {h_2}^{1}) \varepsilon_s(h_3 \cdot 1_A)  \varepsilon_s({{g_1}^{1}}{{g_2}^{1}})\varepsilon_s({h_2}^{0} \cdot y_2) \\
			& \ & \left. \varepsilon_s({{g_2}^{0}}\cdot 1_A)  \varepsilon_s({g_3}\cdot 1_A) \varepsilon_s({t_1}^{1}) 
			\varepsilon_s({t_1}^{0}\cdot 1_A)\varepsilon_s(t_2 \cdot 1_A) \right).
		\end{eqnarray*}
		
		Moreover,
		\begin{eqnarray*}
			& \ & \varepsilon \left( (x \due h)(y \due g)_2 \right) \varepsilon \left((y \due g)_1(z \due t) \right) \\ 
			&=&\varepsilon \left( (x \due h)(y_2 {g_1}^{1} \due g_2)) \varepsilon((y_1 \due g_1^{0})(z \due t) \right) \\
			&=&\varepsilon(x(h_1 \cdot y_2 {g_1}^1) \due h_2 g_2) \varepsilon(y_1({{g_1}^{0}}_{1} \cdot z ) \# {{g_1}^{0}}_{2} t) \\
			&=&\varepsilon(x(h_1 \cdot y_2 {g_1}^1) (( h_2 g_2)^{0} \cdot 1_A) (h_2g_2)^{1} (h_3g_3 \cdot 1_A))  \\
			&\ &\varepsilon(y_1({{g_1}^{0}}_{1} \cdot z ) (({{g_1}^{0}}_{2} t_1)^{0} \cdot 1_A) {{g_1}^{0}}_{2} t_1)^{1}({{g_1}^{0}}_{3} t_2 \cdot 1_A))  \\
			&\stackrel{(\ref{4.6})}{=}&\varepsilon(x\varepsilon_s(h_1 \cdot 1_A) \varepsilon_s(y_2 ) \varepsilon_s({g_1}^1) (( h_2 g_2)^{0} \cdot 1_A) (h_2g_2)^{1} (h_3g_3 \cdot 1_A))  \\
			& \ & \varepsilon(y_1({{g_1}^{0}}_{1} \cdot z )(({{g_1}^{0}}_{2} t_1)^{0} \cdot 1_A) {{g_1}^{0}}_{2} t_1)^{1}({{g_1}^{0}}_{3} t_2 \cdot 1_A))  \\
			&\stackrel{(\ref{4.6})}{=}&\varepsilon(x y_2 (h_1 \cdot 1_A)  \varepsilon_s({g_1}^1) (( h_2 g_2)^{0} \cdot 1_A) (h_2g_2)^{1} (h_3g_3 \cdot 1_A)) \\
			& \ &\varepsilon(y_1({{g_1}^{0}}_{1} \cdot z ) (({{g_1}^{0}}_{2} t_1)^{0} \cdot 1_A) {{g_1}^{0}}_{2} t_1)^{1}({{g_1}^{0}}_{3} t_2 \cdot 1_A))  \\
			&\stackrel{(\ref{4.6})}{=}&\varepsilon(x \varepsilon_s (h_1 \cdot y)  \varepsilon_s({g_1}^1) \varepsilon_s(( h_2 g_2)^{0} \cdot 1_A) \varepsilon_s((h_2g_2)^{1}) \varepsilon_s(h_3g_3 \cdot 1_A) \varepsilon_s({{g_1}^{0}}_{1} \cdot z )\\
			& \ &\varepsilon_s(({{g_1}^{0}}_{2} t_1)^{0} \cdot 1_A) \varepsilon_s(({{g_1}^{0}}_{2} t_1)^{1})\varepsilon_s({{g_1}^{0}}_{3} t_2 \cdot 1_A))  \\
			&{=}& \varepsilon((h_1 \cdot y)_1)\varepsilon(x \varepsilon_s ((h_1 \cdot y)_2 (h_2g_2)^{1}) \varepsilon_s({g_1}^1) \varepsilon_s(( h_2 g_2)^{0} \cdot 1_A)  \varepsilon_s(h_3 \cdot 1_A)\\
			& \ &  \varepsilon_s(g_3 \! \cdot \! 1_A) \varepsilon_s({{g_1}^{0}}_{1} \! \cdot \! z ) \varepsilon_s({{g_1}^{0}}_{2} \! \cdot \! 1_A)\varepsilon_s(t_2 \! \cdot \! 1_A) \varepsilon_s(({{g_1}^{0}}_{3} t_1)^{1})\varepsilon_s(({{g_1}^{0}}_{3} t_2)^{0} \! \cdot \! 1_A)) \\
			&{=}& \varepsilon({h_1}^0 \cdot y_1)\varepsilon(x \varepsilon_s ({h_1}^1 {h_2}^1) (h_3 \cdot y_2{g_2}^{1}) \varepsilon_s({g_1}^1) \varepsilon_s( {h_2}^{0} {g_2}^{0} \cdot 1_A)  \varepsilon_s(h_4 \cdot 1_A) \\
			& \ &\varepsilon_s(g_3 \cdot 1_A) \varepsilon_s({{g_1}^{0}}_{1} \cdot z ) \varepsilon_s(t_2 \cdot 1_A) \varepsilon_s(({{g_1}^{0}}_{2} t_1)^{0} \cdot 1_A)) \varepsilon_s(({{g_1}^{0}}_{2} t_1)^{1})\\
			&\stackrel{(\ref{4.6})}{=}& \varepsilon({h_1}^0 \cdot y_1)\varepsilon(x \varepsilon_s ({h_1}^1 {h_2}^1) \varepsilon_s(h_3 \cdot 1_A) \varepsilon_s(y_2) \varepsilon_s({g_2}^{1}) \varepsilon_s({g_1}^1) \varepsilon_s( {h_2}^{0} \cdot 1_A)   \\
			& \ &\varepsilon_s( {g_2}^{0} \cdot 1_A)  \varepsilon_s(h_4 \cdot 1_A) \varepsilon_s(g_3 \cdot 1_A) \varepsilon_s(({{g_1}^{0}}_{1} \cdot z)_2 ) \varepsilon_s(t_2 \cdot 1_A)  \\
			& \ & \varepsilon_s(({{g_1}^{0}}_{2} t_1)^{0} \cdot 1_A))  \varepsilon_s(({{g_1}^{0}}_{2} t_1)^{1}) \varepsilon_s(({{g_1}^{0}}_{1} \cdot z)_1 )\\
			&\stackrel{\ref{novo weak matched pair}(iv)}{=}& \varepsilon({h_1}^0 \cdot y_1)\varepsilon(x \varepsilon_s({h_1}^1 {h_2}^1 )\varepsilon_s({h_2}^0 \cdot y_2)  \varepsilon_s(h_3 \cdot 1_A) \varepsilon_s({g_1}^1{g_2}^{1} )   \varepsilon_s( {g_2}^{0} \cdot 1_A)    \\
			& \ &\varepsilon_s(g_3 \cdot 1_A) \varepsilon_s({{{g_1}^{0}}_{1}}^{1} {{{g_1}^{0}}_{2}}^{1} ({{{g_1}^{0}}_{3}} \cdot z_2{t_1}^{1})) \varepsilon_s( {{{g_1}^{0}}_{2}}^{0} {t_1}^{0} \cdot 1_A) \\
			& \ & \varepsilon_s(t_2 \cdot 1_A)) \varepsilon({{{g_1}^{0}}_{1}}^{0} \cdot z_1 )\\
			&{=}& \varepsilon({h_1}^0 \cdot y_1)\varepsilon(x \varepsilon_s({h_1}^1 {h_2}^1) \varepsilon_s({h_2}^0 \cdot y_2)  \varepsilon_s(h_3 \cdot 1_A) \varepsilon_s({g_1}^1{g_2}^{1}{g_3}^{1})   \varepsilon_s( {g_3}^{0} \cdot 1_A)    \\
			& \ & \varepsilon_s(g_4 \cdot 1_A) \varepsilon_s({g_1}^{01}{{{g_2}^{0}}_{1}}^{1} ({{{g_2}^{0}}_{2}} \cdot z_2{t_1}^{1})) \varepsilon_s( {{{g_2}^{0}}_{1}}^{0} {t_1}^{0} \cdot 1_A) \varepsilon_s(t_2 \cdot 1_A))\\
			& & \varepsilon({g_1}^{00} \cdot z_1 )\\
			&\stackrel{\ref{lema_basic}(ii)}{=}& \varepsilon({h_1}^0 \cdot y_1)\varepsilon({g_1}^{00} \cdot z_1 )\varepsilon(x \varepsilon_s({h_1}^1 {h_2}^1) \varepsilon_s({h_2}^0 \cdot y_2)  \varepsilon_s(h_3 \cdot 1_A) \varepsilon_s({g_1}^1{g_2}^{1}{g_3}^{1})   \\
			& \ & \varepsilon_s( {g_3}^{0} \cdot 1_A)    \varepsilon_s(g_4 \cdot 1_A) \varepsilon_s({g_1}^{01}{g_2}^{01})\varepsilon_s(z_2) \varepsilon_s({t_1}^{1})\varepsilon_s({g_2}^{00} \cdot 1_A) \\
			&\ & \varepsilon_s({t_1}^{0} \cdot 1_A) \varepsilon_s(t_2 \cdot 1_A)) \\
			&{=}& \varepsilon({h_1}^0 \cdot y_1)\varepsilon({g_1}^{0} \cdot z_1 )\varepsilon(x \varepsilon_s({h_1}^1 {h_2}^1) \varepsilon_s({h_2}^0 \cdot y_2)  \varepsilon_s(h_3 \cdot 1_A) \varepsilon_s({{g_1}^1}_2{{g_2}^{1}}_2{g_3}^{1}) \\
			& \ &\varepsilon_s( {g_3}^{0} \cdot 1_A)    \varepsilon_s(g_4 \cdot 1_A) \varepsilon_s({{g_1}^1}_1{{g_2}^{1}}_1)  \varepsilon_s(z_2) \varepsilon_s({t_1}^{1})\varepsilon_s({g_2}^{0} \cdot 1_A)  \\
			& \ &\varepsilon_s({t_1}^{0} \cdot 1_A) \varepsilon_s(t_2 \cdot 1_A)) \\
			&{=}& \varepsilon({h_1}^0 \cdot y_1)\varepsilon({g_1}^{0} \cdot z_1 )\varepsilon(x \varepsilon_s({h_1}^1 {h_2}^1) \varepsilon_s({h_2}^0 \cdot y_2)  \varepsilon_s(h_3 \cdot 1_A) \varepsilon_s({{g_1}^1}{{g_2}^{1}}) \\
			& \ &\varepsilon_s(g_3 \cdot 1_A)  \varepsilon_s({t_1}^{1})\varepsilon_s({g_2}^{0} \cdot z_2)  \varepsilon_s({t_1}^{0} \cdot 1_A) \varepsilon_s(t_2 \cdot 1_A)).
		\end{eqnarray*}
		
		Then,
		\begin{eqnarray*}
			\varepsilon((x \due h)(y \due g)(z \due t))&=& \varepsilon ((x \due h)(y \due g)_1)\varepsilon ((y \due g)_2(z \due t))\\
			&=& \varepsilon ((x \due h)(y \due g)_2) \varepsilon((y \due g)_1(z \due t)).
		\end{eqnarray*}
		
		Now, we  prove the property for $\Delta(1_A\due 1_H)$. To do this, by Proposition \ref{equivalencia_doum}, we show the following two equivalent equalities:
		$$(x \due h)_1 \otimes \varepsilon_t((x \due h)_2)= \Delta(1_A \due 1_H)(x\due h \otimes 1_A \due 1_H),$$
		$$(x \due h)_1 \otimes \varepsilon_s'((x \due h)_2)= (x\due h \otimes 1_A \due 1_H)\Delta(1_A \due 1_H).$$
		Indeed,
		\begin{eqnarray*}
			& \ &(x \due h)_1 \otimes \varepsilon_t((x \due h)_2)\\
			&=& x_1 \due {h_1}^0 \otimes \varepsilon_t(x_2{h_1}^1 \due h_2) \\
			&=& x_1 \due {h_1}^0 \otimes (1_A \due 1_H)_2\varepsilon((1_A \due 1_H)_1(x_2{h_1}^1 \due h_2)) \\
			&=&x_1 \due {h_1}^0 \otimes {1_A}_2 {{1_H}_1}^1 \due {1_H}_2 \varepsilon({1_A}_1 ({{{1_H}_1}^0}_1 \cdot x_2 {h_1}^1) \due {{{1_H}_1}^0}_2 h_2) \\
			&\stackrel{(\ref{4.6})}{=}& x_1 \due {h_1}^0 \otimes {1_A}_2 {{1_H}_1}^1 \due {1_H}_2 \varepsilon({1_A}_1 (({{{1_H}_1}^0}_2 h_2)^0 \cdot 1_A)\varepsilon_s({{{1_H}_1}^0}_1 \cdot 1_A)  \\
			& \ & \varepsilon_s(x_2{h_1}^1) \varepsilon_s(({{{1_H}_1}^0}_2 h_2)^1) \varepsilon_s({{{1_H}_1}^0}_3 \cdot (h_3 \cdot 1_A))) \\
			&\stackrel{\ref{novo weak matched pair}(iv)}{=}& x_1 \due {h_1}^0 \otimes {1_A}_2 {{1_H}_1}^1 \due {1_H}_2\varepsilon({1_A}_1 (({{{1_H}_1}^0}_2 h_2)^0 \cdot 1_A) \varepsilon_s(x_2{h_1}^1) \varepsilon_s(h_3 \cdot 1_A))\\
			& \ &\varepsilon_s(({{{1_H}_1}^0}_1 \cdot 1_A)_1 ({{{1_H}_1}^0}_1 \cdot 1_A)_2 ({{{1_H}_1}^0}_2 h_2)^1) \\
			&{=}& x_1 \due {h_1}^0 \otimes {1_A}_2 {{1_H}_1}^1 \due {1_H}_2\varepsilon({1_A}_1 (({{{1_H}_1}^0}_2 h_2)^0 \cdot 1_A) \varepsilon_s(x_2{h_1}^1) \varepsilon_s(h_3 \cdot 1_A)\\
			& \ &\varepsilon_s(({{{1_H}_1}^0}_1 \cdot 1_A)_1) \varepsilon_s( ({{{1_H}_1}^0}_1 \cdot 1_A)_2 ({{{1_H}_1}^0}_2 h_2)^1)) \\
			&\stackrel{\ref{troca}}{=}& x_1 \due {h_1}^0 \otimes {1_A}_2 {{1_H}_1}^1 \due {1_H}_2\varepsilon({1_A}_1 ({{{{1_H}_1}^0}_2}^0 {h_2}^0 \cdot 1_A) \varepsilon_s(x_2{h_1}^1) \varepsilon_s(h_3 \cdot 1_A)\\
			& \ &\varepsilon_s({{{{1_H}_1}^0}_1}^0 \cdot {1_A'}_1) \varepsilon_s({{{{1_H}_1}^0}_1}^1  {{{{1_H}_1}^0}_2}^1 ({{{{1_H}_1}^0}_3} \cdot {1_A'}_2 {h_2}^1)))\\
			&\stackrel{\ref{novo weak matched pair}(iv)}{=}& x_1 \due {h_1}^0 \otimes {1_A}_2 {{1_H}_1}^1 \due {1_H}_2 \varepsilon({1_A}_1 \varepsilon_s({{{{{1_H}_1}^0}_1}^0}_2 \cdot 1_A) \varepsilon_s({h_2}^0 \cdot 1_A) \varepsilon_s(x_2{h_1}^1) \\
			& \ &\varepsilon_s(h_3 \cdot 1_A) \varepsilon_s({{{{{1_H}_1}^0}_1}^0} _1\cdot {1_A}) \varepsilon_s({1_A'}_1)\varepsilon_s({{{{1_H}_1}^0}_1}^1) \varepsilon_s(  {{{{1_H}_1}^0}_2}\cdot 1_A) \varepsilon_s({1_A'}_2 {h_2}^1))\\
			&{=}& x_1 \due {h_1}^0 \otimes {1_A}_2 {{1_H}_1}^1 \due {1_H}_2 \varepsilon({1_A}_1 \varepsilon_s({{{{{1_H}_1}^0}_1}^0} \cdot 1_A) \varepsilon_s({h_2}^0 \cdot 1_A)  \\
			& \ & \varepsilon_s(x_2{h_1}^1 {h_2}^1) \varepsilon_s(h_3 \cdot 1_A) \varepsilon_s({{{{1_H}_1}^0}_1}^1) \varepsilon_s(  {{{{1_H}_1}^0}_2}\cdot 1_A))\\
			&{=}& x_1 \due {h_1}^0 \otimes {1_A}_2 {{1_H}_1}^1 {{1_H}_2}^1 \due {1_H}_3 \varepsilon({1_A}_1 \varepsilon_s({{1_H}_1}^{00} \cdot 1_A) \varepsilon_s({h_2}^0 \cdot 1_A) \\
			& \ &\varepsilon_s(x_2{h_1}^1 {h_2}^1)   \varepsilon_s(h_3 \cdot 1_A) \varepsilon_s({{1_H}_1}^{01}) \varepsilon_s(  {{{1_H}_2}^0}\cdot 1_A))\\
			&{=}& x_1 \due {h_1}^0 \otimes {1_A}_2 {{{1_H}_1}^1}_2 {{1_H}_2}^1 \due {1_H}_3 \varepsilon({1_A}_1 \varepsilon_s({{1_H}_1}^0 \cdot 1_A) \varepsilon_s({h_2}^0 \cdot 1_A)  \\
			& \ & \varepsilon_s(x_2{h_1}^1 {h_2}^1) \varepsilon_s(h_3 \cdot 1_A) \varepsilon_s({{{1_H}_1}^{1}}_1) \varepsilon_s(  {{{1_H}_2}^0}\cdot 1_A))\\
			&\stackrel{(\ref{4.6})}{=}& x_1 \due {h_1}^0 \otimes {1_A}_3 {{{1_H}_1}^1}_2 {{1_H}_2}^1 \due {1_H}_3 \varepsilon({1_A}_2 {{{1_H}_1}^1}_1) \varepsilon({1_A}_1\varepsilon_s({{1_H}_1}^0 \cdot 1_A)  \\
			& \ & \varepsilon_s({h_2}^0 \cdot 1_A) \varepsilon_s(x_2{h_1}^1 {h_2}^1) \varepsilon_s(h_3 \cdot 1_A) \varepsilon_s({{{1_H}_2}^0}\cdot 1_A))\\
			&{=}& x_1 \due {h_1}^0 \otimes {1_A}_2 {{{1_H}_1}^1} {{1_H}_2}^1 \due {1_H}_3 \varepsilon({1_A}_1\varepsilon_s({{1_H}_1}^0 \cdot 1_A) \varepsilon_s({h_2}^0 \cdot 1_A) \\
			& \ & \varepsilon_s(x_2{h_1}^1 {h_2}^1) \varepsilon_s(h_3 \cdot 1_A) \varepsilon_s({{{1_H}_2}^0}\cdot 1_A))\\
			&{=}& x_1 \due {{h_1}^0}_1 \otimes {1_A}_2 {{1_H}_1}^1 \due {1_H}_2 \varepsilon({1_A}_1\varepsilon_s({{{1_H}_1}^0}_1 \cdot 1_A) \varepsilon_s({{h_1}^0}_2 \cdot 1_A) \varepsilon_s(x_2{h_1}^1)\\
			& \ & \varepsilon_s(h_2 \cdot 1_A) \varepsilon_s({{{1_H}_1}^0}_2\cdot 1_A))\\
			&\stackrel{\ref{lema_basic}(ii)}{=}& x_1 \due {{h_1}^0}_1 \otimes {1_A}_2 {{1_H}_1}^1 \due {1_H}_2 \varepsilon({1_A}_1\varepsilon_s({{{1_H}_1}^0} \cdot 1_A)  \varepsilon_s(x_2{h_1}^1) \varepsilon_s(h_2 \cdot 1_A) \\
			& \ &\varepsilon_s({{h_1}^0}_2 \cdot 1_A))\\
			&\stackrel{(\ref{4.6})}{=}& x_1 \due {{h_1}^0}_1 \otimes {1_A}_3 {{1_H}_1}^1 \due {1_H}_2 \varepsilon({1_A}_1 x_2{h_1}^1  ({{h_1}^0}_2 \cdot 1_A))\\
			& \ & \varepsilon({1_A}_2 ({{{1_H}_1}^0} \cdot 1_A)    (h_2 \cdot 1_A)) \\
			&\stackrel{(\ref{4.15})}{=}& x \varepsilon_s({1_A}_1) \varepsilon_s({h_1}^1) \varepsilon_s ({{h_1}^0}_2 \cdot 1_A) \due {{h_1}^0}_1 \otimes {1_A}_3 {{1_H}_1}^1 \due {1_H}_2 \\
			& \ &\varepsilon({1_A}_2 ({{{1_H}_1}^0} \cdot 1_A)    (h_2 \cdot 1_A)) \\
			&\stackrel{\ref{lema_basic}(i)}{=}& x {1_A}_1 \varepsilon_s({h_1}^1)  \due {{h_1}^0} \otimes {1_A}_3 {{1_H}_1}^1 \due {1_H}_2  \varepsilon({1_A}_2 ({{{1_H}_1}^0} \cdot 1_A)    (h_2 \cdot 1_A)) \\
			&{=}& x_1 {1_A}_1 \varepsilon_s({h_1}^1)_1 ({{{h_1}^0}_1}^0 \cdot 1_A) \# {{{h_1}^0}_2}^0 \otimes {1_A}_4 {{1_H}_1}^1 \due {1_H}_2  \\
			& \ & \varepsilon({1_A}_3 ({{{1_H}_1}^0} \cdot 1_A)    (h_2 \cdot 1_A)) \varepsilon (x_2 {1_A}_2 \varepsilon_s({h_1}^1)_2 {{{h_1}^0}_1}^1 {{{h_1}^0}_2}^1 ({{{h_1}^0}_3} \cdot 1_A)) \\
			&\stackrel{(\ref{4.15})}{=}& x {1_A}_1 \varepsilon_s({h_1}^1) \varepsilon_s({{{h_1}^0}_1}^1 {{{h_1}^0}_2}^1) ({{{h_1}^0}_1}^0 \cdot 1_A)  \# {{{h_1}^0}_2}^0 \otimes {1_A}_3 {{1_H}_1}^1 \due {1_H}_2 \\ 
			& \ & \varepsilon({1_A}_2 ({{{h}_1}^0}_3 \cdot 1_A)   ({{1_H}_1}^0 \cdot 1_A) (h_2 \cdot 1_A)) \\
			&=& x {1_A}_1 \varepsilon_s({h_1}^1{h_2}^1) \varepsilon_s({h_1}^{01} {{{h_2}^0}_1}^1) ({h_1}^{00} \cdot 1_A)  \# {{{h_2}^0}_1}^0 \otimes {1_A}_3 {{1_H}_1}^1 \due {1_H}_2 \\ 
			& \ & \varepsilon({1_A}_2 ({{{h}_2}^0}_2 \cdot 1_A)   ({{1_H}_1}^0 \cdot 1_A) (h_3 \cdot 1_A)) \\
			&=& x {1_A}_1 \varepsilon_s({{h_1}^1}_2{h_2}^1{h_3}^1) \varepsilon_s({{h_1}^1}_1 {h_2}^{01}) ({h_1}^{0} \cdot 1_A)  \# {h_2}^{00} \otimes {1_A}_3 {{1_H}_1}^1 \due {1_H}_2 \\ 
			& \ & \varepsilon({1_A}_2 ({{{h}_3}^0} \cdot 1_A)   ({{1_H}_1}^0 \cdot 1_A) (h_4 \cdot 1_A)) \\
			&=& x {1_A}_1 \varepsilon_s({{h_1}^1}{h_2}^1{h_3}^1) ({h_1}^{0} \cdot 1_A)  \# {h_2}^{0} \otimes {1_A}_3 {{1_H}_1}^1 \due {1_H}_2 \\ 
			& \ & \varepsilon({1_A}_2 ({{{h}_3}^0} \cdot 1_A)   ({{1_H}_1}^0 \cdot 1_A) (h_4 \cdot 1_A)) \\
			&=& x {1_A}_1 \varepsilon_s({{h_1}^1}) ({{h_1}^{0}}_1 \cdot 1_A)  \# {{h_1}^{0}}_2 \otimes {1_A}_3 {{1_H}_1}^1 \due {1_H}_2 \\ 
			& \ & \varepsilon({1_A}_2 ({{h_1}^{0}}_3 \cdot 1_A)   ({{1_H}_1}^0 \cdot 1_A) (h_2 \cdot 1_A)) \\
			&\stackrel{\ref{novo weak matched pair}(v)}{=}& x {1_A}_1 \varepsilon_s({{h_1}^1}) ({{h_1}^{0}}_1 \cdot 1_A)  \# {{h_1}^{0}}_2  \otimes {1_A}_3 {{1_H}'_1}^1 \varepsilon({{1_H}'_1}^0) \due {1_H}_2 {{1_H}'_2} \\ 
			& \ & \varepsilon({1_A}_2 ({{h_1}^{0}}_3 \cdot 1_A)   ({{1_H}_1} \cdot 1_A) (h_2 \cdot 1_A)) \\
			&\stackrel{(\ref{4.6})}{=}& x {1_A}_1 \varepsilon_s({{h_1}^1}) ({{h_1}^{0}}_1 \cdot 1_A)  \# {{h_1}^{0}}_2 \otimes {1_A}_3 {{1_H}'_1}^1 \varepsilon({{1_H}'_1}^0) \due {1_H}_2 {{1_H}'_2} \\ 
			& \ & \varepsilon({1_A}_2 \varepsilon_s({{h_1}^{0}}_3 \cdot 1_A)   \varepsilon_s({{1_H}_1} \cdot 1_A) (h_2 \cdot 1_A)) \\
			&=& x {1_A}_1 \varepsilon_s({{h_1}^1}) ({{h_1}^{0}}_1 \cdot 1_A)  \# {{h_1}^{0}}_2 \otimes {1_A}_3 {{1_H}'_1}^1 \varepsilon({{1_H}'_1}^0) \due {1_H}_2 {{1_H}'_2} \\ 
			& \ & \varepsilon({1_A}_2 \varepsilon_s({{1_H}_1} {{h_1}^{0}}_3 \cdot 1_A)   (h_2 \cdot 1_A)) \\
			&\stackrel{(\ref{4.12})}{=}& x {1_A}_1 \varepsilon_s({{h_1}^1}) ({{h_1}^{0}}_1 \cdot 1_A)  \# {{h_1}^{0}}_2 \otimes {1_A}_3 {{1_H}'_1}^1 \varepsilon({{1_H}'_1}^0) \due \varepsilon_t({{h_1}^0}_4) {{1_H}'_2} \\ 
			& \ & \varepsilon({1_A}_2 \varepsilon_s({{h_1}^{0}}_3 \cdot 1_A)   (h_2 \cdot 1_A)) \\
			&\stackrel{(\ref{4.15})}{=}& x {1_A}_1 \varepsilon_s({{h_1}^1}) ({{h_1}^{0}}_1 \cdot 1_A)  \# {{h_1}^{0}}_2 \otimes {1_A}_3 {{1_H}'_1}^1 \varepsilon({{1_H}'_1}^0) \due \varepsilon_t({{h_1}^0}_3) {{1_H}'_2} \\
			& \ &\varepsilon({1_A}_2  (h_2 \cdot 1_A)).
		\end{eqnarray*}
		
		On the other hand,
		\begin{eqnarray*}
			& \ &\Delta( 1_A \due 1_H) (x \due h \otimes 1_A \due 1_H)\\
			& =& ({1_A}_1 \due {{1_H}_1}^0)(x \due h) \otimes ({1_A}_2 {{1_H}_1}^1 \due {1_H}_2)(1_A \due 1_H)\\
			& = & ({1_A}_1 ({{{1_H}_1}^0}_1 \cdot x)_1 \underline{\#} ({{{1_H}_1}^0}_2 h_1)^0) \otimes ({1_A}_3 {{1_H}_1}^1 \due {1_H}_2) \\
			& \ & \varepsilon({1_A}_2 ({{{1_H}_1}^0}_1 \cdot x)_2  ({{{1_H}_1}^0}_2 h_1)^1 ({{{1_H}_1}^0}_3 h_2 \cdot 1_A))\\
			& = & ({1_A}_1 ({{{{1_H}_1}^0}_1}^0 \cdot x_1) \underline{\#} {{{{1_H}_1}^0}_2}^0 {h_1}^0) \otimes ({1_A}_3 {{1_H}_1}^1 \due {1_H}_2)\\
			& \ &\varepsilon({1_A}_2 {{{{1_H}_1}^0}_1}^1 {{{{1_H}_1}^0}_2}^1 ({{{1_H}_1}^0}_3 \cdot x_2 {h_1}^1) ({{{1_H}_1}^0}_4 h_2  \cdot 1_A))\\
			&\stackrel{\ref{novo weak matched pair}(iv)}{=}& ({1_A}_1 ({{{{1_H}_1}^0}_1}^0 \cdot x_1) \underline{\#} {{{{1_H}_1}^0}_2}^0 {h_1}^0) \otimes ({1_A}_3 {{1_H}_1}^1 \due {1_H}_2) \\
			& \ &\varepsilon({1_A}_2 \varepsilon_s({{{{1_H}_1}^0}_1}^1 {{{{1_H}_1}^0}_2}^1) \varepsilon_s({{{1_H}_1}^0}_3 \cdot 1_A) \varepsilon_s(x_2 {h_1}^1) \varepsilon_s({{{1_H}_1}^0}_4   \cdot 1_A) \varepsilon_s(h_2 \cdot 1_A))\\
			&=& ({1_A}_1 ({{{1_H}_1}^{00}}_1 \cdot x_1) \underline{\#} {{{1_H}_1}^{00}}_2  {h_1}^0) \otimes ({1_A}_3 {{1_H}_1}^1 {{1_H}_2}^1 \due {1_H}_3)\varepsilon({1_A}_2 \varepsilon_s({{1_H}_1}^{01})   \\
			& \ &\varepsilon_s(x_2 {h_1}^1) \varepsilon_s({{{1_H}_2}^0}   \cdot 1_A) \varepsilon_s(h_2 \cdot 1_A))\\
			&=& ({1_A}_1 ({{{1_H}_1}^{0}}_1 \cdot x_1) \underline{\#} {{{1_H}_1}^{0}}_2  {h_1}^0) \otimes ({1_A}_3 {{{1_H}_1}^1} {{{1_H}_2}^1}  \due {1_H}_3)\varepsilon({1_A}_2 \varepsilon_s({{{1_H}_2}^0}   \cdot 1_A)  \\
			& \ & \varepsilon_s(x_2 {h_1}^1) \varepsilon_s(h_2 \cdot 1_A))\\
			&=& ({1_A}_1 ({{{1_H}_1}^{0}}_1 \cdot x_1) \underline{\#} {{{1_H}_1}^{0}}_2  {h_1}^0) \otimes ({1_A}_4 {{{1_H}_1}^1}   \due {1_H}_2)\varepsilon({1_A}_2 ({{{1_H}_1}^0}_3   \cdot 1_A) x_2 {h_1}^1)  \\
			& \ &  \varepsilon({1_A}_3(h_2 \cdot 1_A))\\
			&\stackrel{(\ref{4.15})}{=}& ({1_A}_1 \varepsilon_s(x_2 {h_1}^1) \varepsilon_s(({{{1_H}_1}^0}_3   \cdot 1_A)) ({{{1_H}_1}^{0}}_1 \cdot x_1) \underline{\#} {{{1_H}_1}^{0}}_2  {h_1}^0) \otimes ({1_A}_3 {{{1_H}_1}^1}   \due {1_H}_2)  \\
			& \ & \varepsilon({1_A}_2(h_2 \cdot 1_A))\\
			&\stackrel{\ref{lema_basic}(i)}{=}& ({1_A}_1 \varepsilon_s(x_2 {h_1}^1)  ({{{1_H}_1}^{0}}_1 \cdot x_1) \underline{\#} {{{1_H}_1}^{0}}_2  {h_1}^0) \otimes ({1_A}_3 {{{1_H}_1}^1}   \due {1_H}_2)  \varepsilon({1_A}_2(h_2 \cdot 1_A))\\
			&\stackrel{\ref{novo weak matched pair}(v)}{=}& ({1_A}_1 \varepsilon_s(x_2 {h_1}^1)  ({{1_H}_1} \cdot x_1) \underline{\#} {{1_H}_2} {h_1}^0) \otimes ({1_A}_3 {{{1_H}'_1}^1} \varepsilon({{{1_H}'_1}^0})  \due {1_H}_3{{1_H}'_2}) \\
			& \ & \varepsilon({1_A}_2(h_2 \cdot 1_A))\\
			&\stackrel{\ref{lema_2_1_9_Glauber}}{=}& ({1_A}_1 \varepsilon_s( {h_1}^1)\varepsilon_s(x_2)x_1  ({{1_H}_1} \cdot 1_A) \underline{\#} {{1_H}_2} {h_1}^0) \otimes ({1_A}_3 {{{1_H}'_1}^1} \varepsilon({{{1_H}'_1}^0})  \due {1_H}_3{{1_H}'_2}) \\
			& \ & \varepsilon({1_A}_2(h_2 \cdot 1_A))\\
			&{=}& ({x1_A}_1 \varepsilon_s( {h_1}^1)({{1_H}_1}{{h_1}^0}_1  \cdot 1_A) \# {{1_H}_2} {{h_1}^0}_2) \otimes ({1_A}_3 {{{1_H}'_1}^1} \varepsilon({{{1_H}'_1}^0})  \due {1_H}_3{{1_H}'_2}) \\
			& \ & \varepsilon ({1_A}_2(h_2 \cdot 1_A))\\
			&\stackrel{(\ref{4.12})}{=}& (x {1_A}_1 \varepsilon_s({{h_1}^1}) ({{h_1}^{0}}_1 \cdot 1_A)  \# {{h_1}^{0}}_2 ) \otimes ({1_A}_3 {{1_H}'_1}^1 \varepsilon({{1_H}'_1}^0) \due \varepsilon_t({{h_1}^0}_3) {{1_H}'_2})\\
			& \ & \varepsilon({1_A}_2  (h_2 \cdot 1_A)).
		\end{eqnarray*}
		
		Analogously, we obtain the identity 
		$$(x \due h)_1 \otimes \varepsilon_s'((x \due h)_2)= (x\due h \otimes 1_A \due 1_H)\Delta(1_A \due 1_H).$$
		Therefore, $A\overline{\underline{\#}} H$ is a weak bialgebra.
	\end{proof}

	\subsection{The structure of weak Hopf algebra}
	
	\
	
	In the previous section, we constructed the weak bialgebra $A \due H$ from a compatible weak matched pair $(H,A)$.
	Now, if $H$ and $A$ are weak Hopf algebras, it is natural ask when we can endow such bialgebra with a weak Hopf structure, under some compatibility conditions for the antipode maps.
	To do this, we need the following result:
	
	\begin{lem}\label{3traingulos}
		Let $(H,A)$ be a compatible weak matched pair. Then for all $a, x \in A$ and $h \in H$, it holds $x_1 ({h_1}^0 \cdot 1_A) \varepsilon(a x_2{h_1}^1 (h_2 \cdot 1_A))= x_1 (h^0 \cdot 1_A) \varepsilon(a x_2 h^1)$.    
	\end{lem}
	\begin{proof}
		Indeed, 
		\begin{eqnarray*}
			& \ & x_1 ({h_1}^0 \cdot 1_A) \varepsilon(a x_2{h_1}^1 (h_2 \cdot 1_A))\\
			&=& x_1 ({h_1}^0 \cdot 1_A)\varepsilon(x_2{h_1}^1) \varepsilon(a x_3 (h_2 \cdot 1_A))\\
			&=& x_1 ({h_1}^0 \cdot 1_A)\varepsilon(h_2)\varepsilon(x_2{h_1}^1) \varepsilon(a x_3 (h_3 \cdot 1_A))\\
			&=&  x_1({{h^0}_1}^{0}\cdot 1_A)\varepsilon({h^0}_2)\varepsilon(x_2{{h^0}_1}^{1})\varepsilon({h}^1)\varepsilon(a x_3 (h_3 \cdot 1_A))\\
			&=&  x_1({h_1}^{00}\cdot 1_A)\varepsilon({h_2}^0)\varepsilon(x_2{h_1}^{01})\varepsilon({h_1}^1{h_2}^2)\varepsilon(a x_3 (h_3 \cdot 1_A))\\
			&=&  x_1({h_1}^0\cdot 1_A)\varepsilon({h_2}^0)\varepsilon(x_2{{h_1}^1}_1)\varepsilon({{h_1}^1}_2{h_2}^2)\varepsilon(a x_3 (h_3 \cdot 1_A))\\
			&=& x_1({h_1}^0\cdot 1_A)\varepsilon({h_2}^0)\varepsilon(x_2{h_1}^1{h_2}^2) \varepsilon(a x_3 (h_3 \cdot 1_A))\\
			&=& x_1({h_1}^0\cdot 1_A)\varepsilon({h_2}^0)\varepsilon(x_2{h_1}^1{h_2}^1(h_3 \cdot 1_A))\varepsilon(x_3a)\\
			&\stackrel{\eqref{**}}{=}& x_1(h_1\cdot 1_A)_1\varepsilon({h_2}^0)\varepsilon(x_2(h_1\cdot 1_A)_2{h_2}^1)\varepsilon(x_3a)\\
			&=& x_1(h_1\cdot 1_A)\varepsilon({h_2}^0)\varepsilon(x_2{h_2}^1)\varepsilon(x_3a)\\
			&\stackrel{\ref{obs_ordem}}{=}& x_1({h^0}_1\cdot 1_A)\varepsilon({h^0}_2)\varepsilon(ax_2{h_2}^1)\\
			&=& x_1 (h^0 \cdot 1_A) \varepsilon(a x_2 h^1).
		\end{eqnarray*}
	\end{proof}
	
	\begin{thm}\label{antipoda}
		Let $(H,A)$ be a  compatible weak matched pair of weak Hopf algebras. If
		\begin{equation} \label{condicao_7}
			(1_{H1}h^0\cdot 1_A)\otimes S_A(h^1)\varepsilon_s({1_{H2}}^1)\otimes {1_{H2}}^0=\varepsilon_t({1_{H1}}^0h^0\cdot 1_A) \otimes \varepsilon_t(h^1)\varepsilon_s({1_{H2}}^1){1_{H1}}^1\otimes {1_{H2}}^0
		\end{equation}
		and
		\begin{equation}\label{condicao_8}
			\varepsilon_s({h_1}^1{h_2}^1)(S_H({h_1}^0)\cdot 1_A)\otimes S_H({h_2}^0)h_3= \varepsilon_s(h^1{1_{H2}}^1)\varepsilon_s(h^0{1_{H1}}\cdot 1_A)\otimes {1_{H2}}^0
		\end{equation}
		hold for all $h,g\in H$, then $A\underline{\overline{\#}}H$ is a weak Hopf algebra with antipode given by
		$$S(x\underline{\overline{\#}}h)=(1_A\underline{\overline{\#}} S_H(h^0))(S_A(xh^1)\underline{\overline{\#}}1_H).$$
	\end{thm}
	
	\begin{proof}
		By Theorem \ref{principal}, we just need to verify the three conditions that refer to the antipode map, namely 
		\begin{eqnarray}
			\label{antipoda_1}    (x\overline{\underline{\#}} h)_1 \ S((x\overline{\underline{\#}} h)_2) \ = \ \varepsilon_t(x\overline{\underline{\#}} h) \\
			\label{antipoda_2}   S((x\overline{\underline{\#}} h)_1) \  (x\overline{\underline{\#}} h)_2 \ =\ \varepsilon_s(x\overline{\underline{\#}} h) \\
			\label{antipoda_3}    S((x\overline{\underline{\#}} h)_1) \ (x\overline{\underline{\#}} h)_2 \ S((x\overline{\underline{\#}} h)_3) \ = \ S(x\overline{\underline{\#}} h).
		\end{eqnarray}
		
		For condition \eqref{antipoda_1}, we have
		\begin{eqnarray*}
			& \ & (x\overline{\underline{\#}} h)_1 \ S((x\overline{\underline{\#}} h)_2) \\
			& = & (x_1 \due {h_1}^0) \ S(x_2 {h_1}^1 \due h_2) \\
			& = & (x_1 \due {h_1}^0) (1_A \due S_H ({h_2}^0))( S_A(x_2{h_1}^1{h_2}^1) \due 1_H) \\
			& = & (x_1 \due {h^0}_1)(1_A\due S_H({h^0}_2))(S_A(x_2h^1) \due 1_H) \\
			& = & (x_1({h^0}_1 \cdot 1_A) \due {h^0}_2S_H({h^0}_3))(S_A(x_2h^1)\due 1_H) \\
			& = & (x_1({h^0}_1 \cdot 1_A) \due \varepsilon_t({h^0}_2))(S_A(x_2h^1)\due 1_H) \\
			& \stackrel{\eqref{4.12}}{=} & (x_1(1_{H1}{h^0} \cdot 1_A) \due 1_{H2})(S_A(x_2h^1)\due 1_H) \\
			& = & x_1(1_{H1}{h^0} \cdot 1_A) (1_{H2} \cdot S_A(x_2h^1)) \due 1_{H3} \\
			& = & x_1(1_{H1} \cdot 1_A) (1_{H2}h^0 \cdot 1_A)( 1_{H3} \cdot S_A(x_2h^1)) \due 1_{H4} \\
			& \stackrel{\ref{lema_2_1_9_Glauber}}{=} & (1_{H1} \cdot x_1) (1_{H3}h^0 \cdot 1_A)( 1_{H2} \cdot S_A(x_2)S_A(h^1)) \due 1_{H4} \\
			& = & ( 1_{H1} \cdot x_1S_A(x_2)S_A(h^1))(1_{H2}h^0 \cdot 1_A) \due 1_{H3} \\
			& \stackrel{\ref{lema_2_1_9_Glauber}}{=} & \varepsilon_t(x)S_A(h^1)(1_{H1} \cdot 1_A)(1_{H2}h^0 \cdot 1_A) \due 1_{H3} \\
			& = & \et(x)S_A(h^1)(1_{H1}h^0 \cdot 1_A) \due 1_{H2} \\
			& = & \et(x)S_A(h^1)(1_{H1}h^0 \cdot 1_A)\es({1_{H2}}^1(1_{H3} \cdot 1_A)) \underline{\#} {1_{H2}}^0 \\
			& = & \et(x)S_A(h^1)\es({1_{H3}}^1)(1_{H1} \cdot 1_A)(1_{H2}h^0 \cdot 1_A)\es({1_{H4}} \cdot 1_A) \underline{\#} {1_{H3}}^0 \\
			&  \stackrel{\ref{lema_basic} (i)}{=}  & \et(x)S_A(h^1)\es({1_{H2}}^1)(1_{H1}h^0 \cdot 1_A) \underline{\#} {1_{H2}}^0 \\
			& \stackrel{\eqref{condicao_7}}{=}& \et(x)\et(h^1)\es({1_{H2}}^1){1_{H1}}^1\et({1_{H1}}^0h^0 \cdot 1_A) \underline{\#} {1_{H2}}^0.
		\end{eqnarray*}
		On the other hand,
		\begin{eqnarray*}
			& \ & \et(x \due h ) \\
			& = & \varepsilon((1_A \due 1_H)_1 ( x \due h)) ( 1_A \due 1_H)_2 \\
			& = & \varepsilon((1_{A1}\due {1_{H1}}^0)(x \due h)) 1_{A2} {1_{H1}}^1 \due 1_{H2} \\
			& = &  \varepsilon (1_{A1}({{1_{H1}}^0}_1 \cdot x) \due {{1_{H1}}^0}_2h) 1_{A2}{1_{H1}}^1 \due 1_{H2} \\
			& = &  \varepsilon \left(1_{A1}({{1_{H1}}^0}_1 \cdot x) (({{1_{H1}}^0}_2h_1)^0 \cdot 1_A) ({{1_{H1}}^0}_2h_1)^1 ({{1_{H1}}^0}_3h_2 \cdot 1_A)\right) \\
			& \ &1_{A2}{1_{H1}}^1 \due 1_{H2} \\
			& \stackrel{\eqref{troca}, \ref{weak matched pair}(iv), \eqref{4.6}}{=} &  \varepsilon ({{{1_{H1}}^0}_1}^0 \cdot x_1) \ \varepsilon (1_{A1}({{{1_{H1}}^0}_2}^0{h_1}^0 \cdot 1_A) {{{1_{H1}}^0}_1}^1 {{{1_{H1}}^0}_2}^1 (h_2 \cdot 1_A) \\
			& \ & ({{1_{H1}}^0}_3 \cdot x_2{h_1}^1) )  1_{A2}{1_{H1}}^1 \due 1_{H2} \\
			& = &  \varepsilon ({{1_{H1}}^{00}} \cdot x_1) \ \varepsilon (1_{A1}({{{1_{H2}}^0}_1}^0{h_1}^0 \cdot 1_A) ( {{{1_{H1}}^{01}}} {{{1_{H2}}^0}_1}^1) (h_2 \cdot 1_A)  \\
			& \ & ({{1_{H2}}^0}_2 \cdot x_2{h_1}^1) ) 1_{A2}{1_{H1}}^1 {1_{H2}}^1 \due 1_{H3} \\
			& = &  \varepsilon ({{1_{H1}}^{00}} \cdot x_1) \ \varepsilon (1_{A1}({1_{H2}}^{00} {h_1}^0 \cdot 1_A)  {{{1_{H1}}^{01}}} {{1_{H2}}^{01}} (h_2 \cdot 1_A)  \\
			& \ & ({1_{H3}}^0 \cdot x_2{h_1}^1) ) 1_{A2}{1_{H1}}^1 {1_{H2}}^1{1_{H3}}^1 \due 1_{H4} \\
			& = &  \varepsilon ({{1_{H1}}^{0}} \cdot x_1) \ \varepsilon (1_{A1}({1_{H2}}^{0} {h_1}^0 \cdot 1_A)  {{{1_{H1}}^{1}}_1} {{{1_{H2}}^1}_1} (h_2 \cdot 1_A)  \\
			& \ & ({1_{H3}}^0 \cdot x_2{h_1}^1) ) 1_{A2}{{1_{H1}}^1}_2 {{1_{H2}}^1}_2 {1_{H3}}^1 \due 1_{H4} \\
			& = &  \varepsilon ({{1_{H1}}^{0}} \cdot x_1) \ \varepsilon \left(1_{A1}({1_{H2}}^{0} {h_1}^0 \cdot 1_A)  (h_2 \cdot 1_A) ({1_{H3}}^0 \cdot x_2{h_1}^1) \right) \\
			& & 1_{A2}{{1_{H1}}^1} {{1_{H2}}^1} {1_{H3}}^1 \due 1_{H4} \\
			& \stackrel{\eqref{4.6}}{=} &  \varepsilon (\varepsilon_s({{{1_{H1}}^{0}}_1} \cdot x_1)) \ \varepsilon \left(1_{A1}({{1_{H1}}^{0}}_2 {h_1}^0 \cdot 1_A)  (h_2 \cdot 1_A) ({1_{H2}}^0 \cdot x_2{h_1}^1) \right) \\
			& & 1_{A2}{{1_{H1}}^1} {{1_{H2}}^1} \due 1_{H3} \\
			& \stackrel{\ref{weak matched pair}(iv)}{=} &  \varepsilon (\varepsilon_s(x_1)\varepsilon_s({{{1_{H1}}^{0}}_1} \cdot 1_A)) \ \varepsilon (1_{A1}\varepsilon_s(x_2)\varepsilon_s(h_1^1)\varepsilon_s({{1_{H1}}^{0}}_2 {h_1}^0 \cdot 1_A) \\
			& &   \varepsilon_s(h_2 \cdot 1_A) ({{1_{H1}}^0}_3 \cdot 1_A) ) 1_{A2}{{1_{H1}}^1} \due 1_{H2} \\
			& \stackrel{\eqref{4.6}, \ \ref{weak matched pair}(iv)}{=} &  \varepsilon(1_{A1} x \varepsilon_s({{1_{H1}}^{0}}_1 \cdot 1_A) \varepsilon_s({{1_{H1}}^{0}}_2 \cdot 1_A) \varepsilon_s({{1_{H1}}^{0}}_3 \cdot 1_A)  \varepsilon_s({h_1}^1) \\
			& & \varepsilon_s(h_2 \cdot 1_A) \varepsilon_s({{h_{1}}^0} \cdot 1_A) ) \ 1_{A2}{{1_{H1}}^1} \due 1_{H2} \\
			& \stackrel{\eqref{4.6}, \ \ref{lema_basic}(ii)}{=} &  \varepsilon\left(1_{A1} x {h_1}^1 ({{1_{H1}}^{0}} \cdot 1_A)   (h_2 \cdot 1_A) ({{h_{1}}^0} \cdot 1_A) \right)  1_{A2}{{1_{H1}}^1} \due 1_{H2} \\ 
			& = &  \varepsilon\left(x_1 ({{h_{1}}^0} \cdot 1_A)\right) \ \varepsilon\left( 1_{A1} x_2 {h_1}^1 ({{1_{H1}}^{0}} \cdot 1_A)   (h_2 \cdot 1_A)  \right)  1_{A2}{{1_{H1}}^1} \due 1_{H2} \\ 
			& \stackrel{\ref{3traingulos}}{=} &  \varepsilon\left(x_1 ({h^0} \cdot 1_A)\right) \ \varepsilon\left( 1_{A1} x_2 {h}^1 ({{1_{H1}}^{0}} \cdot 1_A) \right)  1_{A2}{{1_{H1}}^1} \due 1_{H2} \\
			& = &  \varepsilon\left(1_{A1} x h^1 ({h^0} \cdot 1_A) ({{1_{H1}}^{0}} \cdot 1_A) \right)  1_{A2}{{1_{H1}}^1} \due 1_{H2}  \\
			& = &  \varepsilon\left(1_{A1} x h^1 ({h^0} \cdot 1_A) ({{1_{H1}}^{0}} \cdot 1_A) \right)  1_{A2}{{1_{H1}}^1} \varepsilon_s\left( {1_{H2}}^1 ({1_{H3}} \cdot 1_A) \right)\underline{\#} {1_{H2}}^0  \\
			& = &  \varepsilon_t\left(x h^1 ({h^0} \cdot 1_A) ({{1_{H1}}^{0}} \cdot 1_A) \right)  {{1_{H1}}^1} 1_{A1} \ \varepsilon\left( 1_{A2} {1_{H2}}^1 ({1_{H3}} \cdot 1_A) \right) \\ 
			& \ &\left({{1_{H2}}^{0}}_1 \cdot 1_A\right) \#  {{1_{H2}}^0}_2  \\
			& = &   \varepsilon_t\left(x h^1 ({h^0} \cdot 1_A) ({{1_{H1}}^{0}} \cdot 1_A) \right)  {{1_{H1}}^1} 1_{A1} \ \varepsilon\left( 1_{A2} {1_{H2}}^1 {1_{H3}}^1 ({1_{H4}} \cdot 1_A) \right)\\
			& \ & \left({{1_{H2}}^{0}} \cdot 1_A\right) \# {{1_{H3}}^0}  \\
			& = &   \varepsilon_t\left(x h^1 ({h^0} \cdot 1_A) ({{1_{H1}}^{0}} \cdot 1_A) \right)  {{1_{H1}}^1} 1_{A1} \left({{1_{H2}}^{0}} \cdot 1_A\right) \\
			& & \varepsilon\left( 1_{A2} {1_{H2}}^1  ({1_{H3}} \cdot 1_A) {1_{H4}}^1 \right)  \# {{1_{H4}}^0}  \\
			& \stackrel{ \ref{3traingulos}}{=} &   \varepsilon_t\left(x h^1 ({h^0} \cdot 1_A) ({{1_{H1}}^{0}} \cdot 1_A) \right)  {{1_{H1}}^1} \varepsilon_s\left({1_{H2}}^1  {1_{H3}}^1 \right) \left({{1_{H2}}^{0}} \cdot 1_A\right) \#  {{1_{H3}}^0}  \\
			& = &   \varepsilon_t\left(x h^1 ({h^0} \cdot 1_A) ({{1_{H1}}^{0}} \cdot 1_A) \right)  {{1_{H1}}^1} \ \varepsilon_s\left({1_{H2}}^1 \right) \underline{\#} {{1_{H2}}^0}  \\
			& \stackrel{ \ref{weak matched pair}(iv) }{=} &   \varepsilon_t(x) \ \varepsilon_t(h^1)  \ \varepsilon_s({{1_{H2}}^{1}})  \ \varepsilon_t\left({{1_{H1}}^{0}}{h^0} \cdot 1_A \right)  {{1_{H1}}^1} \underline{\#} {{1_{H2}}^0}.
		\end{eqnarray*}
		So, the equality \eqref{antipoda_1} holds.
		
		Now, we verify the condition \eqref{antipoda_2}. By one hand, we have
		\begin{eqnarray*}
			& \ & S\left( (x\overline{\underline{\#}} h)_1 \right) \ (x\overline{\underline{\#}} h)_2 \\
			& = & S \left( x_1 \due {h_1}^0 \right) \ (x_2 {h_1}^1 \due h_2) \\
			& = & \left( 1_A \due S_H\left({h_1}^{00}\right)\right) \left( S_A\left(x_1{h_1}^{01}\right) \due 1_H \right) \left(x_2 {h_1}^1 \due h_2 \right) \\
			& = & \left( 1_A \due S_H\left({h_1}^{00}\right)\right) \left( S_A\left(x_1{h_1}^{01}\right)   \left(1_{H1} \cdot x_2 {h_1}^1 \right)\due 1_{H2}h_2 \right) \\
			& \stackrel{\ref{lema_2_1_9_Glauber} (ii)}{=} & \left( 1_A \due S_H\left({h_1}^{00}\right)\right) \left(  S_A\left({h_1}^{01}\right)S_A\left(x_1 \right) x_2 {h_1}^1 \left(1_{H1} \cdot 1_A \right) \due 1_{H2}h_2 \right) \\
			& = & \left( 1_A \due S_H\left({h_1}^{0}\right)\right) \left(  \varepsilon_s\left(x\right)S_A\left({{h_1}^{1}}_1\right){{h_1}^1}_2 \left(1_{H1} \cdot 1_A \right)\left(1_{H2}h_2 \cdot 1_A \right) \overline{\#} 1_{H3}h_3 \right) \\
			& = & \left( 1_A \due S_H\left({h_1}^{0}\right)\right) \left(  \varepsilon_s\left(x\right) \varepsilon_s\left({h_1}^{1}\right) \left(1_{H1}h_2 \cdot 1_A \right) \overline{\#} 1_{H2}h_3 \right) \\
			& = & \left( 1_A \due S_H\left({h_1}^{0}\right)\right) \left(  \varepsilon_s\left(x\right) \varepsilon_s\left({h_1}^{1}\right)  \due h_2 \right) \\
			& = & \left( S_H\left({{h_1}^{0}}_2 \right)\cdot   \varepsilon_s\left(x\right) \varepsilon_s\left({h_1}^{1}\right) \right) \due S_H \left( {{h_1}^{0}}_1 \right)h_2  \\
			& \stackrel{\ref{lema_basic} (iii)}{=} & \varepsilon_s\left(x {h_1}^{1}\right) \left( S_H\left({{h_1}^{0}}_2 \cdot  1_A \right)\right)  \due S_H \left( {{h_1}^{0}}_1 \right)h_2 \\
			& = & \varepsilon_s\left(x {h_1}^{1}\right)  \left(\varepsilon_t\left(\left(S_H\left({{h_1}^{0}}\right)_2\right) \right) \cdot  1_A \right)  \due S_H \left({{h_1}^{0}}\right)_1 h_2 \\
			& \stackrel{\eqref{4.12}}{=} & \varepsilon_s\left(x\right) \varepsilon_s\left({h_1}^{1}\right)  \left(1_{H2} \cdot 1_A \right)  \due 1_{H1} S_H \left({{h_1}^{0}}\right) h_2 \\
			& = & \varepsilon_s\left(x\right) \varepsilon_s\left({h_1}^{1}\right)  \left(1_{H3} \cdot 1_A \right) \left(1_{H1}\left(S\left({h_1}^0\right)h_2\right)_1 \cdot 1_A \right)  \overline{\#} 1_{H2} \left(S_H \left({{h_1}^{0}}\right) h_2\right)_2 \\
			& = & \varepsilon_s\left(x\right) \varepsilon_s\left({h_1}^{1}\right) \due S_H \left({{h_1}^{0}}\right) h_2 \\
			& = & \varepsilon_s\left(x\right) \varepsilon_s\left({h_1}^{1}\right) \left(S_H \left({{h_1}^{0}}\right)_1 h_2 \cdot 1_A \right) \overline{\#} S_H \left({{h_1}^{0}}\right)_2 h_3 \\
			& = & \varepsilon_s\left(x\right) \varepsilon_s\left({h_1}^{1}\right) \left(S_H \left({{h_1}^{0}}_1\right) \cdot \left(h_2 \cdot 1_A \right) \right) \overline{\#} S_H \left({{h_1}^{0}}_2\right) h_3 \\
			& = & \varepsilon_s\left(x\right) \varepsilon_s\left({h_1}^{1}\right) \left(S_H \left({{h_1}^{0}}_1\right) \cdot 1_A \right) \left(S_H \left({{h_1}^{0}}_2\right) h_2 \cdot 1_A \right) \overline{\#} S_H \left({{h_1}^{0}}_3\right) h_3 \\
			& = & \varepsilon_s \left( x \right) \varepsilon_s \left( {h_1}^{1} \right) \left( S_H \left( {{h_1}^{0}}_1 \right) \cdot 1_A \right)  \due S_H \left( {{h_1}^{0}}_2 \right) h_2 \\
			& = & \varepsilon_s\left(x\right) \varepsilon_s\left({h_1}^{1}{h_2}^{1}\right) \left(S_H \left({{h_1}^{0}}\right) \cdot 1_A \right)  \due S_H \left({{h_2}^{0}}\right) h_3 \\
			& \stackrel{\eqref{condicao_8}}{=} & \varepsilon_s\left(x\right) \varepsilon_s\left(h^{1} {1_{H2}}^{1}\right) \varepsilon_s\left({{h}^0} 1_{H1} \cdot 1_A \right)  \due {1_{H2}}^0.
		\end{eqnarray*}
		On the other hand,
		\begin{eqnarray*}
			& \ & \es(x \due h ) \\
			& = & (1_A \due 1_H)_1  \ \varepsilon\left(( x \due h) ( 1_A \due 1_H)_2\right) \\
			& = & 1_{A1}\due {1_{H1}}^0 \ \varepsilon\left( \left(x \due h\right) \left( 1_{A2} {1_{H1}}^1 \due 1_{H2} \right) \right)\\
			& = & 1_{A1}\due {1_{H1}}^0 \ \varepsilon\left( x \left( h_1 \cdot 1_{A2} {1_{H1}}^1 \right) \due h_2 1_{H2} \right)\\
			& = & 1_{A1}\due {1_{H1}}^0 \ \varepsilon( x \left( h_1 \cdot {1_{H1}}^1 \right) \left( h_2 \cdot 1_{A2}\right)_2 \left( \left(h_3 1_{H2}\right)^0 \cdot 1_A\right) \left(h_3 1_{H2}\right)^1 \\
			& & \left(h_4 1_{H3} \cdot 1_A \right) ) \varepsilon \left( \left( h_2 \cdot 1_{A2} \right)_1 \right)\\
			& \stackrel{\ref{troca}}{=} & 1_{A1}\due {1_{H1}}^0 \ \varepsilon( x \left( h_1 \cdot {1_{H1}}^1 \right) \left( {h_3}^0 {1_{H2}}^0 \cdot 1_{A}\right) {h_2}^1 {h_3}^1 \left(h_4 \cdot 1_{A3} {1_{H2}}^1 \right) \\
			& & \left(h_5 1_{H3} \cdot 1_A \right) ) \varepsilon \left( {h_2}^0 \cdot 1_{A2} \right)\\
			& \stackrel{\ref{weak matched pair} (iv), \eqref{4.6}}{=} & 1_{A1}\due {1_{H1}}^0 \ \varepsilon( 1_{A2} x \es\left({1_{H1}}^1\right) \es\left( h_1 \cdot 1_A \right) \es\left( {h_3}^0 \cdot 1_{A}\right) \es\left({1_{H2}}^0 \cdot 1_A \right) \\
			& &  \es\left({h_2}^1 {h_3}^1\right) 
			\es\left(h_4 \cdot 1_A\right) \es\left({1_{H2}}^1 \right) \es\left(h_5 \cdot 1_A \right) \es\left(1_{H3} \cdot 1_A\right) \\
			& &  \es\left({h_2}^0 \cdot 1_A\right) )\\
			& \stackrel{\ref{lema_basic} (ii)}{=} & 1_{A1}\due {1_{H1}}^0 \ \varepsilon( 1_{A2} x \es\left({1_{H1}}^1 {1_{H2}}^1 \right) \es\left( h_1 \cdot 1_A \right) \es\left( {h_2}^0 \cdot 1_{A}\right)  \\
			& &\es\left({1_{H2}}^0 \cdot 1_A \right) \es\left({h_2}^1\right)   \left. \es\left(h_3 \cdot 1_A\right) \es\left(1_{H3} \cdot 1_A\right) \right)\\
			& \stackrel{\ref{lema_basic} (ii)}{=} & 1_{A1}\due {{1_{H1}}^0}_1 \ \varepsilon( 1_{A2} x \es\left({1_{H1}}^1 \right) \es\left( h_1 \cdot 1_A \right) \es\left( {h_2}^0 \cdot 1_{A}\right)\\
			& &\es\left({{1_{H1}}^0}_2 \cdot 1_A \right) \es\left({h_2}^1\right) \es\left(1_{H2} \cdot 1_A\right) )\\
			& \stackrel{\eqref{4.15}}{=} &  \es\left(x\right) \es\left({1_{H1}}^1 \right) \es\left( h_1 \cdot 1_A \right) \es\left( {h_2}^0 \cdot 1_{A}\right) \es\left({h_2}^1\right) \es\left({{1_{H1}}^0}_3 \cdot 1_A \right)  \\
			&  & \es\left(1_{H2} \cdot 1_A\right)   \left({{1_{H1}}^0}_1 \cdot 1_A \right) \overline{\#} {{1_{H1}}^0}_2  \\
			& \stackrel{\eqref{lema_basic}(i)}{=} &  \es\left(x\right) \es\left({1_{H1}}^1 \right) \es\left( h_1 \cdot 1_A \right) \es\left( {h_2}^0 \cdot 1_{A}\right) \es\left({h_2}^1\right) \left({{1_{H1}}^0}_1 \cdot 1_A \right) \\
			& & \es\left(1_{H2} \cdot 1_A\right) \overline{\#} {{1_{H1}}^0}_2  \\
			& = &  \es\left(x\right) \es\left( {h_1}^1 \left( h_2 \cdot 1_A \right)\right) \es\left( {h_1}^0 \cdot 1_{A}\right) \es\left( {1_{H1}}^1 {1_{H2}}^1 \left(1_{H3} \cdot 1_A\right) \right)\\
			& & \left({1_{H1}}^0 \cdot 1_A \right)   \overline{\#} {1_{H2}}^0  \\
			& \stackrel{\ref{3traingulos}, \ \eqref{4.20}}{=} &  \es\left(x\right) \es\left( h^1\right) \es\left( h^0 \cdot 1_A\right) \es\left( {1_{H1}}^1 \left(1_{H2} \cdot 1_A\right) \right) \es\left( {1_{H3}}^1 \right)\left({1_{H1}}^0 \cdot 1_A \right)   \\
			& &  \overline{\#} {1_{H3}}^0  \\
			& \stackrel{\ref{weak matched pair} (iv)}{=} &  \es\left(x\right) \es\left( h^1 {1_{H2}}^1 {1_{H3}}^1\right) \es\left( h^0 1_{H1} \cdot 1_A\right) \left({1_{H2}}^0 \cdot 1_A \right)   \overline{\#} {1_{H3}}^0  \\
			& = &  \es\left(x\right) \es\left( h^1 {1_{H2}}^1 \right) \es\left( h^0 1_{H1} \cdot 1_A \right) \left({{1_{H2}}^0}_1 \cdot 1_A \right)   \overline{\#} {{1_{H2}}^0}_2  \\
			& = & \varepsilon_s\left(x\right) \varepsilon_s\left(h^{1} {1_{H2}}^{1}\right) \varepsilon_s\left({{h}^0} 1_{H1} \cdot 1_A \right)  \due {1_{H2}}^0.
		\end{eqnarray*}
		Thus, the equality \eqref{antipoda_2} holds.
		
		Finally, we check the condition \eqref{antipoda_3}. Then,
		\begin{eqnarray*}
			& &  S\left( \left(x \due h\right)_1 \right) \ \left(x \due h\right)_2 \ S\left(\left(x \due h\right)_3\right)  \\
			& = &  S\left( \left(x \due h\right)_1 \right) \ \et\left(\left(x \due h\right)_2\right)  \\
			& = &  S\left(x_1 \due {h_1}^0  \right) \ \et\left(x_2 {h_1}^1 \due h_2 \right)  \\
			& = &  S\left(x_1 \due {h_1}^0  \right) \ \left( \et\left(x_2 {h_1}^1\right) S_A\left({h_2}^1\right) \left(1_{H1}{h_2}^0 \cdot 1_A \right) \due 1_{H2} \right)  \\
			& = &  \left(1_A \due S_H \left({h_1}^{00}\right)  \right) \ \left(S_A\left(x_1 {h_1}^{01}\right) \due 1_H \right) \ ( \et\left(x_2 {h_1}^1\right) S_A\left({h_2}^1\right) \\
			& & \left(1_{H1}{h_2}^0 \cdot 1_A \right) \due 1_{H2} )  \\
			& = &  \left(1_A \due S_H \left({h_1}^{00}\right)  \right) \ (S_A\left(x_1 {h_1}^{01}\right) ( {1_{H1}'} \cdot \et\left(x_2 {h_1}^1\right) S_A\left({h_2}^1\right) \\
			& & \left(1_{H1}{h_2}^0 \cdot 1_A \right) ) \due {1_{H2}'} 1_{H2} )  \\
			& \stackrel{\ref{lema_2_1_9_Glauber} (ii)}{=} &  \left(1_A \due S_H \left({h_1}^{00}\right)  \right) \ (S_A\left(x_1 {h_1}^{01}\right) \ \et\left(x_2 {h_1}^1\right) S_A\left({h_2}^1\right) \\
			& & \left( {1_{H1}'} \cdot  \left(1_{H1}{h_2}^0 \cdot 1_A \right) \right) \due {1_{H2}'} 1_{H2} )  \\
			& = &  \left(1_A \due S_H \left({h_1}^{00}\right)  \right) \ (S_A\left(x_1 {h_1}^{01}\right) \ \et\left(x_2 {h_1}^1\right) S_A\left({h_2}^1\right) \left(1_{H1}{h_2}^0 \cdot 1_A \right) \due \\
			& & 1_{H2} )  \\
			& = &  \left(1_A \due S_H \left({h_1}^0\right)  \right) \ (S_A\left(x_1 {{h_1}^1}_1\right) \ \et\left(x_2 {{h_1}^1}_2\right) S_A\left({h_2}^1\right) \left(1_{H1}{h_2}^0 \cdot 1_A \right) \due \\
			& & 1_{H2} )  \\
			& = &  \left(1_A \due S_H \left({h_1}^0\right)  \right) \ \left(S_A\left(x {h_1}^1 \right) \ S_A\left({h_2}^1\right) \left(1_{H1}{h_2}^0 \cdot 1_A \right) \due 1_{H2} \right)  \\
			& = &  \left(1_A \due S_H \left({h^0}_1\right)  \right) \ \left(S_A\left(x h^1 \right) \ \left(1_{H1}{h^0}_2 \cdot 1_A \right) \due 1_{H2} \right)  \\
			& = &  \left( S_H \left({h^0}_1\right)  \cdot  \left(S_A\left(x h^1 \right) \ \left(1_{H1}{h^0}_3 \cdot 1_A \right) \right) \right)  \due S_H \left({h^0}_2\right) 1_{H2}  \\
			& = &  \left( S_H \left({h^0}_1\right)  \cdot S_A\left(x h^1 \right) \right) \left(S_H\left({h^0}_2\right) 1_{H1}{h^0}_3 \cdot 1_A \right)  \due S_H \left({h^0}_4\right) 1_{H2}  \\
			& = &  \left( S_H \left({h^0}_1\right)  \cdot S_A\left(x h^1 \right) \right) \left(S_H\left({h^0}_2\right)_1 \ 1_{H1}{h^0}_3 \cdot 1_A \right)  \due S_H \left({h^0}_2\right)_2 1_{H2}  \\
			& = &  \left( S_H \left({h^0}_1\right)  \cdot S_A\left(x h^1 \right) \right) \left(S_H\left({h^0}_2\right) {h^0}_3 \cdot 1_A \right)  \due S_H \left({h^0}_4\right) \\
			& = &  \left( S_H \left({h^0}_1\right)  \cdot S_A\left(x h^1 \right) \right) \left( \es\left({h^0}_2\right) \cdot 1_A \right)  \due S_H \left({h^0}_3\right) \\
			& \stackrel{\eqref{4.13}}{=} &  \left( S_H \left({h^0}_1 1_{H2}\right)  \cdot S_A\left(x h^1 \right) \right) \left( 1_{H1} \cdot 1_A \right)  \due S_H \left({h^0}_2\right) \\
			& = &  \left( S_H \left(1_{H2}\right) S_H \left({h^0}_1 \right)  \cdot S_A\left(x h^1 \right) \right) \left( 1_{H1} \cdot 1_A \right)  \due S_H \left({h^0}_2\right) \\
			& \stackrel{\eqref{4.39}}{=} &  \left(1_{H1} \  S_H \left(1_{H2}\right) S_H \left({h^0}_1 \right)  \cdot S_A\left(x h^1 \right) \right) \left( 1_{H3} \cdot 1_A \right)  \due S_H \left({h^0}_2\right) \\
			& = &  \left(1_{H1} \  S_H \left(1_{H2}\right) S_H \left({h^0}_1 \right)  \cdot S_A\left(x h^1 \right) \right)  \due S_H \left({h^0}_2\right) \\
			& = &  \left( \et\left(1_H \right) S_H \left({h^0}_1 \right)  \cdot S_A\left(x h^1 \right) \right)  \due S_H \left({h^0}_2\right) \\
			& = &  \left( 1_A \due S_H \left(h^0 \right)  \right) \left( S_A\left(x h^1 \right) \due 1_H \right)   \\
			& = & S \left( x \due h \right).
		\end{eqnarray*}
		Therefore, the equality \eqref{antipoda_3} holds, and so the proof is complete.
	\end{proof}

	\section{Examples and properties}
	
	In this section we present some examples of weak Hopf algebras $A\due H$ as well as some properties about the theory of integrals.
	
	First, recall that a compatible weak matched pair $(H,A)$ gives rise to a weak bialgebra $A \due H$ (see Definitions \ref{novo weak matched pair} and \ref{weak matched pair}, and Theorem \ref{principal}).
	Furthermore, if $H$ and $A$ are Hopf algebras and some conditions hold, then $A \due H$ can be endowed with a weak Hopf algebra structure (see Theorem \ref{antipoda}).
	
	So, in the following, we show that the examples of weak matched pairs presented in Section \ref{sec_weak_matched_pair} are, in fact, compatible matched pairs.
	Moreover, since all the weak bialgebras  we are dealing with are weak Hopf algebras, we also verify that compatibility conditions for the antipode maps hold,  obtaining weak Hopf algebras.
	
	\begin{ex}
		Consider the abelian weak matched pair $(\mathcal{H}^G,\mathcal{H}^G)$ of Example \ref{exemplo_par_combinado_HG}.
		Then $\mathcal{H}^G\due \mathcal{H}^G$ is a weak Hopf algebra.
		Moreover, it is not a (usual) Hopf algebra.
		
		Indeed, the two items of Definition \ref{weak matched pair} are clear, then $(\mathcal{H}^G,\mathcal{H}^G)$  is a compatible matched pair and so  $\mathcal{H}^G\due \mathcal{H}^G$ is a weak bialgebra.
		Also, for $h \in \mathcal{H}^G$, is a routine computation to conclude that both sides of equality \eqref{condicao_7} are reduced to $\frac{1}{|G|^2} \sum_{g,a} hga \otimes g^{-1} \otimes a^{-1}$, and similarly both sides of equality \eqref{condicao_8} are $\frac{1}{|G|} \sum_{g} hg \otimes g^{-1}.$
		
		Hence, by  Theorem \ref{antipoda} we conclude that $\mathcal{H}^G\due \mathcal{H}^G$ is a weak Hopf algebra.
		Moreover, since $x \due h = \frac{1}{|G|}\sum_{g \in G} xhg \# g^{-1}$, we get on the one hand
		\begin{equation*}
			\Delta (1 \due 1)  =  1_1 \due 1_1' \otimes 1_2 1_2' \due 1_3'=   \frac{1}{|G|^3} \sum_{g, a , b \in G} g a \# a^{-1} \otimes g^{-1}b \# b^{-1}  
		\end{equation*}
		and, on the other hand,
		\begin{equation*}
			1 \due 1 \otimes 1 \due 1 = \frac{1}{|G|^2} \sum_{g, h \in G} g \# g^{-1} \otimes h \# h^{-1}.
		\end{equation*}
		Then, $\Delta (1 \due 1)  \neq  1 \due 1 \otimes 1 \due 1$ and so this weak Hopf algebra is not a (usual) Hopf algebra.
		In fact, it can be concluded that $\mathcal{H}^G \due \mathcal{H}^G \simeq \mathcal{H}^{\frac{G \times G}{N}}$, where $N = \{(g,g^{-1}), \ | g \ \in G \}.$
	\end{ex} 
	
	\begin{ex} \label{ex_smashlambdaz}
		Assume that $(H,A)$ of Example \ref{exemplo_par_combinado_lambda_z} is an abelian weak matched pair, that is, $H$ is cocommutative, $A$ is commutative, $H$ acts on $A$ via $\lambda \in H^*$ and $A$ coacts on $H$ via $z \in A$.
		Then $A\due H$ is a weak Hopf algebra.
		Moreover, we have $A\due H = Az \otimes  H_\lambda$, where $Az = \{xz \ | \ x \in A \}$ and $H_\lambda=\{ \lambda(h_1)h_2 \ | \ h\in H \}$.
		Since both $Az$ and $H_\lambda$ are Hopf algebras, we conclude that $A \due H$ is a (usual) Hopf algebra.
		
		Indeed, clearly $(H,A)$  is a compatible matched pair and so  $A\due H$ is a weak bialgebra.
		Now, in order to verify conditions \eqref{condicao_7} and \eqref{condicao_8}, note that, in this case, $\varepsilon_s(z)= \varepsilon_t(z)=S(z)=z, z^2=z$, and $\lambda(h)=\lambda(\varepsilon_t(h))=\lambda(\varepsilon_s(h))$.
		Then, for any $h \in H$, 
		\begin{eqnarray*}
			(1_{H1}h^0\cdot 1_A)\otimes S_A(h^1)\varepsilon_s({1_{H2}}^1)\otimes {1_{H2}}^0 & = & \lambda(1_{H1}h) 1_A\otimes S_A(z)\varepsilon_s(z)\otimes {1_{H2}} \\
			& = & \lambda(1_{H1}h) 1_A\otimes z \otimes {1_{H2}}
		\end{eqnarray*}
		and
		\begin{eqnarray*}
			& & \varepsilon_t({1_{H1}}^0h^0\cdot 1_A) \otimes \varepsilon_t(h^1)\varepsilon_s({1_{H2}}^1){1_{H1}}^1\otimes {1_{H2}}^0 \\
			&=& \varepsilon_t(\lambda({1_{H1}}h)1_A) \otimes \varepsilon_t(z)\varepsilon_s(z)z\otimes 1_{H2}\\
			& = & \lambda({1_{H1}}h)1_A \otimes z \otimes 1_{H2}.
		\end{eqnarray*}
		Thus \eqref{condicao_7} holds.
		Similarly, for equality \eqref{condicao_8},
		\begin{eqnarray*}
			\varepsilon_s({h_1}^1{h_2}^1)(S_H({h_1}^0)\cdot 1_A)\otimes S_H({h_2}^0)h_3& = & \varepsilon_s(z^2)\lambda(S_H(h_1) ) 1_A \otimes S_H(h_2)h_3\\
			& = & z\lambda(S_H(h_1) )\otimes \varepsilon_s(h_2)\\
			& = & z\lambda(\varepsilon_t(S_H(h_1)))\otimes \varepsilon_s(h_2)\\
			& \stackrel{(\ref{4.34})}{=} & z\lambda(\varepsilon_t(\varepsilon_s(h_1)))\otimes \varepsilon_s(h_2)\\
			& \stackrel{(\ref{4.13})}{=} & z\lambda(h 1_{H1}) \otimes 1_{H2}
		\end{eqnarray*}
		and
		\begin{eqnarray*}
			\varepsilon_s(h^1{1_{H2}}^1) \varepsilon_s(h^0{1_{H1}}\cdot 1_A) \otimes {1_{H2}}^0 & = & \varepsilon_s(z^2) \varepsilon_s(\lambda(h 1_{H1} ) 1_A) \otimes 1_{H2}\\ 
			& = & z \lambda(h 1_{H1}) \otimes 1_{H2}.
		\end{eqnarray*}

		Hence, by  Theorem \ref{antipoda} we have that $A \due H$ is a weak Hopf algebra.
		Moreover, since $x \due h = x \varepsilon_s({h_1}^{1}{h_2}^{1}( h_3 \cdot 1_A))({h_1}^{0} \cdot 1_A) \# {h_2}^{0} = x \varepsilon_s(z^2\lambda(h_3))\lambda(h_1) \# h_2 = x z \# \lambda(h_1)h_2$,
		it is easy to conclude that  $A\due H = Az \otimes  H_\lambda$.
		
		In particular, consider $A=H= \mathring{\cup}_{i=1}^{n} H_i$ the finite disjoint union of Hopf algebras, with action and coaction given by $\lambda=\varepsilon_i$ and $z=1_{H_j}$, respectively.
		Then, $A\due H=H_j\otimes H_i$.
	\end{ex}
	
	\begin{ex}
		Consider $(H',A')$ the weak matched pair of Example \ref{ex_kap2}.
		Also, assume that $A$ is commutative and $H$ is cocommutative. 
		Then, since $x \due h= x \# h,$ $\1_{A'} \due h =  \mathbb{e}_{A} \# h,$ $x \due \1_{H'} = x \# \mathbb{e}_H$ and $\1_{A'} \due \1_{H'} = \mathbb{e}_{A} \#  \mathbb{e}_{H},$
		for all $x \in A,$ $h \in H$, we conclude $A' \due H'=A\# H$ is a (usual) Hopf algebra.
	\end{ex}
	
	\medskip
	
	Finally, let us recall the concept of integrals for weak Hopf algebras and the weak version of the well known Maschke Theorem (see \cite{Gabintegral}).

	\begin{defi} \label{def_integral} Consider $H$ a weak Hopf algebra. A left (resp. right) integral in $H$ is an element  $\alpha\in H$ such that $h \alpha=\varepsilon_t(h)\alpha$  (resp. $\alpha h=\alpha\varepsilon_s(h))$, for  all  $h\in H$. We denote $\int_\ell^H$ (resp. $\int_r^H$) the space of left (resp. right) integrals in $H$.
	\end{defi}
	
	\begin{pro} \label{pro_semisimple}
		Let $H$ be a finite weak Hopf algebra, then the following conditions are equivalent:
		\begin{itemize}
			\item[(i)] $H$ is semisimple,
			\item[(ii)] $H$ is separable,
			\item[(iii)] there exists $\alpha\in\int_\ell^H$ such that $\varepsilon_t(\alpha)=1_H$.
		\end{itemize}
	\end{pro}

	The following examples and properties were inspired in Hopf algebra theory, in particular, in what was developed in \cite{matchedpair}, regarding partial matched pair.
	
	\begin{ex}
		Consider $\mathcal{H}^G$ the weak Hopf algebra of Example \ref{exemplodogrupo}. Since $\varepsilon_t(g)=\varepsilon_s(g)=g$, for all $g \in G$, we get  $\int_{\ell}^{\mathcal{H}^G} = \int_{r}^{\mathcal{H}^G} = \mathcal{H}^G$ and so $dim \left(\int_{\ell}^{\mathcal{H}^G} \right)= dim \ \mathcal{H}^G = |G|$. 
		Moreover, $\mathcal{H}^G$ is a semisimple weak Hopf algebra, since 
		$\varepsilon_t(1_{\mathcal{H}^G})=1_{\mathcal{H}^G}$.
	\end{ex}
	
	\begin{ex}
		Let $H= \mathring{\cup}_{i=1}^{n} H_i$ be a finite disjoint union of Hopf algebras.
		Since, for $h \in H_i$, $\varepsilon_t(h)=\varepsilon_i(h)1_{H_i}$, we have
		that $\alpha \in H$ is a left integral in $H$ if and only if $\alpha=\sum_{i=1}^n \alpha_i$, where $\alpha_i \in \int_\ell^{H_i}$.
		Moreover
		$\int_\ell^{H}=\Bbbk \left(\mathring{\cup}_{i=1}^{n} \int_\ell^{H_i}\right) = \sum_{i=1}^n \int_\ell^{H_i}$.
		Furthermore, $H$ is semisimple if and only if $H_i$ is semisimple, for some $1\leq i \leq n$.
	\end{ex}
	
	\begin{ex} Let $H$ be a Hopf algebra and consider $H'$ the weak Hopf algebra of the Example \ref{ex_kapla}.
		Since, $\varepsilon_t(h)=\varepsilon(h) \mathbb{e}$, for all $h \in H$, we conclude $\int_\ell^{H'} = \Bbbk (\1 - \mathbb{e}) + \int_\ell^H$.
		Therefore, $H'$ is semisimple if and only if $H$ semisimple. 
	\end{ex}
	
	\begin{lem} Let $(H,A)$ be a weak matched pair.
		If $\alpha\in \int_\ell^A$, then $h\cdot\alpha\in\int_\ell^A$, for all $h\in H$.
	\end{lem}
	\begin{proof} Consider $\alpha\in \int_\ell^A$ and $h\in H$, then
		\begin{eqnarray*}
			x(h\cdot\alpha)&= &  x(\varepsilon_t(h_1)h_2\cdot\alpha)\\
			&= & x(\varepsilon_t(h_1)\cdot (h_2\cdot\alpha))\\
			&\stackrel{\ref{lema_2_1_9_Glauber}}{=} & x(\varepsilon_t(h_1)\cdot 1_A)(h_2\cdot\alpha)\\
			&\stackrel{\ref{lema_2_1_9_Glauber}}{=} & (\varepsilon_t(h_1)\cdot x)(h_2\cdot\alpha)\\
			&= & (h_1\cdot\alpha)(\varepsilon_t(h_2)\cdot x)\\
			&= & (h_1\cdot\alpha)(h_2S(h_3)\cdot x)\\
			&= & (h_1\cdot\alpha)(h_2\cdot (S(h_3)\cdot x))\\
			&= & h_1\cdot ((S(h_2)\cdot x)\alpha)\\
			& = & h_1\cdot (\varepsilon_t(S(h_2)\cdot x)\alpha)\\
			&\stackrel{\ref{vart_acao}}{=} &  h_1\cdot (\varepsilon_t(S(h_2)\cdot 1_A)\varepsilon_t(x)\alpha)\\
			&= & h_1\cdot (\varepsilon_t(S(h_2)\cdot 1_A)\varepsilon_t(\varepsilon_t(x))\alpha)\\
			&= &  h_1\cdot (\varepsilon_t(S(h_2)\cdot \varepsilon_t(x))\alpha)\\
			&\stackrel{\ref{def_integral}}{=} & h_1\cdot ((S(h_2)\cdot \varepsilon_t(x))\alpha)\\
			&= & (h_1\cdot\alpha)(h_2S(h_3)\cdot \varepsilon_t(x))\\
			&= & (h_1\cdot\alpha)(\varepsilon_t(h_2)\cdot \varepsilon_t(x))\\
			&\stackrel{\ref{lema_2_1_9_Glauber}}{=} & \varepsilon_t(x) (h_1\cdot\alpha)(\varepsilon_t(h_2)\cdot 1_A)\\
			&= & \varepsilon_t(x) (\varepsilon_t(h_1)\cdot(h_2\cdot\alpha))\\
			&= & \varepsilon_t(x) (h\cdot\alpha).
		\end{eqnarray*}
		Therefore, $h\cdot\alpha\in\int_\ell^A$.
	\end{proof}
	
	For the next results, we consider the weak Hopf algebra structure in $A \due H$ as in Theorem \ref{antipoda}.
	
	\begin{pro} Let $A\due H$ be a weak Hopf algebra, and $\alpha \in \int_\ell^A$ and $t\in \int_\ell^H$ such that
		\begin{equation}\label{cond_int}
			(h\cdot\alpha)\varepsilon_s({t_1}^1(t_2\cdot 1_A))\otimes {t_1}^0=(h^0\cdot 1_A)\varepsilon_s(h^{1})\alpha\varepsilon_s({t_1}^1(t_2\cdot 1_A))\otimes {t_1}^0,
		\end{equation} for all $h\in H$. 
		Then, $\alpha\due t\in \int_\ell^{A\due H}$. 
	\end{pro}
	\begin{proof} We need to check that $\left( x \due h \right)\left( \alpha \due t \right)=\varepsilon_t\left( x \due h \right)\left( \alpha \due t \right)$, for all $x\in A$ and $h\in H$. Indeed, 
		\begin{eqnarray*}
			\left( x \due h \right)\left( \alpha \due t \right) &= &  x(h_1\cdot \alpha) \due h_2t\\
			&\stackrel{\ref{vart_acao}}{=} & \varepsilon_t(x)(h_1\cdot \alpha) \due \varepsilon_t(h_2)t\\
			&\stackrel{(\ref{4.12})}{=} &  \varepsilon_t(x)(1_{H1}h\cdot \alpha) \due 1_{H2}t \\
			&= & \varepsilon_t(x)(1_{H1}\cdot (h\cdot \alpha)) \due 1_{H2}t\\
			&\stackrel{\ref{lema_2_1_9_Glauber}}{=} & \varepsilon_t(x)(h\cdot\alpha)(1_{H1}\cdot 1_A) \due 1_{H2}t\\
			&= &  \varepsilon_t(x)(h\cdot\alpha) \due t\\
			&= &  \varepsilon_t(x)(h\cdot\alpha)\varepsilon_s({t_1}^1(t_2\cdot 1_A))\underline{\#}{t_1}^0\\
			&\stackrel{(\ref{cond_int})}{=} & \varepsilon_t(x)\varepsilon_s(h^{1})(h^{0}\cdot 1_A)\alpha\varepsilon_s({t_1}^1(t_2\cdot 1_A))\underline{\#}{t_1}^0\\
			&\stackrel{(\ref{4.34})}{=} & \varepsilon_t(x)\varepsilon_t(S(h^{1}))(h^{0}\cdot 1_A)\alpha\due t\\
			&\stackrel{\eqref{4.7}}{=} & \varepsilon_t(x)S(h^{1})(h^{0}\cdot 1_A)\alpha( 1_{H1}\cdot 1_A)\due \varepsilon_t(1_{H2})t\\
			&\stackrel{\ref{lema_2_1_9_Glauber}}{=} & \varepsilon_t(x)S(h^{1})( 1_{H1}\cdot (h^{0}\cdot 1_A)\alpha)\due \varepsilon_t(1_{H2})t\\
			&= & \varepsilon_t(x)S(h^{1})( 1_{H1}h^{0}\cdot 1_A)(1_{H2}\cdot\alpha)\due 1_{H3}t\\
			&= & \left(\varepsilon_t(x)S(h^{1})( 1_{H1}h^{0}\cdot 1_A) \due 1_{H2}\right)\left( \alpha \due t \right)\\
			&= & \varepsilon_t\left( x \due h \right)\left( \alpha \due t \right).
		\end{eqnarray*}
	\end{proof}
	
	\begin{cor}
		Let $A\due H$ be given by Theorem \ref{antipoda}. Then, $A\due H$ is semisimple if and only if there exists $\alpha\in\int_\ell^{A}$ and $t\in\int_\ell^{H}$ such that $$\varepsilon_t(\alpha)S(t^1)(t^0\cdot 1_A)\due 1_H=1_A\due 1_H.$$
	\end{cor}
	
	\begin{ex} Consider the (weak) Hopf algebra $A\due H=Az\otimes H_{\lambda}$ of the Example \ref{ex_smashlambdaz}. Note that if $\alpha \in \int_\ell^A$ and $t\in \int_\ell^H$, then $\alpha\due t=\alpha z\# \lambda(t_1)t_2\in \int_\ell^{A\due H}$. Indeed, 
		\begin{eqnarray*}
			(h\cdot\alpha)\varepsilon_s({t_1}^1(t_2\cdot 1_A))\otimes {t_1}^0 &= & \lambda(h)\alpha \varepsilon_s(z)\otimes\lambda(t_2)t_1\\
			&= &  (h^0\cdot 1_A)\varepsilon_s(h^{1})\alpha\varepsilon_s({t_1}^1(t_2\cdot 1_A))\otimes {t_1}^0,
		\end{eqnarray*}
		that is, the condition (\ref{cond_int}) is satisfied. Moreover, $$\varepsilon_t(z\alpha \# \lambda(t_1)t_2)=\varepsilon(\alpha z \# \lambda(t_1)t_2)(1_{Az}\due 1_{H_{\lambda}})=\varepsilon(\alpha z)\lambda(t) (1_{Az}\due 1_{H_{\lambda}}).$$
		Therefore, $Az\otimes H_{\lambda}$ is semisimple if only if $Az$ and $H_{\lambda}$ are semisimple. Note that, in this case, there exists $t\in \int_\ell^H$ such that $\lambda(t)\neq 0$, reobtaining \cite{matchedpair}*{Corollary 4.11}.
	\end{ex}


\begin{thebibliography}{99}
		\bibitem{Agore} A.-L. Agore, G. Militaru, \emph{Bicrossed Products, Matched Pair Deformations and the Factorization Index for Lie Algebras}, Symmetry, Integrability and Geometry: Methods and Applications SIGMA 10 (2014), 065, 16 pages.
		
		\bibitem{Nicolas_intro} N. Andruskiewitsch, S. Natale, \emph{Braided {H}opf algebras arising from matched pairs of groups}, Journal of Pure and Applied Algebra 182 (2) (2003), 119-149.
		
		\bibitem{matchedpair} D. Azevedo, G. Martini, A. Paques, L. Silva, \emph{{H}opf algebras arising from partial (co)actions},
		Journal of Algebra and Its Applications 20 (01) (2021), 2140006.
		
		\bibitem{Batista_why} E. Batista, \emph{Partial actions: what they are and why we care}, Bulletin of the Belgian Mathematical Society Simon Stevin 24 (2017) 35-71.
		
		\bibitem{Gabintegral} G. Bohm, N. Florian, S. Korn\'el, \emph{Weak Hopf algebras: I. Integral theory and c-structure}, Journal of Algebra 221 (2) (1999), 385-438.
		
		\bibitem{Bohmexemplo} G. Bohm, J. G\'omes-Torrecillas, \emph{On the Double Crossed Product of Weak Hopf Algebras}, Contemporary Mathematics 585 (2013), 153-174.
		
		\bibitem{Caenepeel_} S. Caenepeel, E. De Groot, \emph{Modules over Weak Entwining Structures}, Contemporary Mathematics 267 (2000), 31-54. 
		
		\bibitem{Felipe} F. Castro, A. Paques, G. Quadros, A. Sant'Ana, \emph{Partial actions of weak Hopf algebras: smash products, globalization and Morita theory}, Journal of Pure and Applied Algebra 29 (2015), 5511-5538.
		
		\bibitem{Kaplansky} Z. Chebel, A. Makhlouf. \emph{Kaplansky's Type Constructions for Weak Bialgebras and Weak Hopf Algebras}, J. Gen. Lie Theory Appl. 9 (S1) (2015), 1-9.
		
		\bibitem{weak_smash_coproduct} G. Fonseca, E. Fontes, G. Martini, \emph{A (partial) weak smash coproduct}, S\~ao Paulo Journal of Mathematical Sciences 17 (2023), 555-594.
		
		\bibitem{Majid} S. Majid, \emph{Foundations of quantum group theory}, Cambridge University Press, 1995.
		
		\bibitem{muller} M. Muller, H. M. P. Pollastri, J. Plavnik, \emph{On bicrossed product of fusion categories and exact factorizations}, (2024). Arxiv Preprint 2405.10207. https://arxiv.org/abs/2405.10207
		
		\bibitem{Sonia_Natale} S. Natale, \emph{Crossed actions of matched pairs of groups on tensor categories}, Tohoku Mathematical Journal 68 (3) (2016), 377-405.
		
		\bibitem{Ricardo} R. L. D. Santos, \emph{Extens\~oes de Ore e Algebras de Hopf Fracas}, PhD Thesis, UFRGS, Porto Alegre (2017).
		
		\bibitem{Singer} W. Singer, \emph{Extension theory for connected Hopf algebras}, Journal of Algebra 21 (1972), 1-16.
		
		
		\bibitem{Takeuchi}  M. Takeuchi, \emph{Matched pairs of groups and bismash products of Hopf algebras}, Communications in Algebra 9 (1981), 841-882.
		
		\bibitem{zhang} T. Zhang, L. Zhang, \emph{Extending structures for Zinbiel algebras}, (2023). Arxiv Preprint 2203.15692. https://arxiv.org/abs/2203.15692
		
		\bibitem{Wang} Yu. Wang, L. Yu. Zhang, \emph{The Structure Theorem for Weak Module Coalgebras}, Mathematical Notes 88 (1) (2010), 3-17.
	\end{thebibliography}
\end{document}